\documentclass[a4paper,11pt]{amsart}

\usepackage[utf8]{inputenc}
\usepackage[english]{babel}
\usepackage{amsthm,amsmath,amssymb}
\usepackage{hyperref}
\usepackage{enumitem}
\setlist[itemize]{label=---}
\usepackage[margin=2cm]{geometry}

\newcommand{\ts}{\textstyle}
\newcommand{\op}{{\mathord{\mathrm{op}}}}
\newcommand{\id}{\mathord{\mathrm{id}}}
\newcommand{\Tr}{\operatorname{Tr}}
\newcommand{\qTr}{\operatorname{qTr}}
\newcommand{\qtr}{\operatorname{qtr}}
\renewcommand{\dim}{\operatorname{dim}}
\newcommand{\qdim}{\operatorname{dim_{\text{$q$}}}}
\newcommand{\Cred}[1]{C^*_r(#1)}
\newcommand{\PAP}{$\mathrm{PAP}$}
\newcommand{\PAPh}{$\mathrm{PAP}_h$}
\newcommand{\conv}{\operatorname{conv}}
\def\ii{{\bf i}}
\def\Sp{\operatorname{Sp}}
\def\GL{\operatorname{GL}}
\def\Prob{\operatorname{Prob}}
\def\Probh{\mathop{\Prob_h}}
\def\Probc{\mathop{\Prob^c}}
\def\Probhc{\mathop{\Prob_h^c}}
\def\Irr{\operatorname{Irr}}
\def\Fix{\operatorname{Fix}}
\def\Corep{\operatorname{Corep}}
\newcommand{\Hom}{\operatorname{Hom}}
\def\AD{\operatorname{AD}} 
\def\ad{\mathop{\mathrm{ad}}\nolimits}
\def\acts{\curvearrowright}
\newcommand\astop{\mathop{\mathord{\ast}^\op}}
\def\fU{\mathop{\mathcal{U}}}
\def\fO{\mathop{\mathcal{O}}}
\def\fP{\mathcal{P}}
\def\N{\mathbb{N}}
\def\Z{\mathbb{Z}}
\def\C{\mathbb{C}}
\def\R{\mathbb{R}}
\def\F{\mathbb{F}}
\def\G{\mathbb{G}}
\def\H{\mathbb{H}}

\theoremstyle{plain}
\newtheorem{thm}{Theorem}[section]
\newtheorem{lem}[thm]{Lemma}
\newtheorem{prop}[thm]{Proposition}
\newtheorem{cor}[thm]{Corollary}

\theoremstyle{definition}
\newtheorem{defn}[thm]{Definition}
\newtheorem{rem}[thm]{Remark}
\newtheorem{eg}[thm]{Example}

\title[$C^*$-simplicity and boundary actions of discrete quantum groups]{$C^*$-simplicity and boundary actions of \\ discrete quantum groups}

\author{Benjamin Anderson-Sackaney}
\address{Benjamin Anderson-Sackaney, Department of Mathematics and Statistics, Faculty of Science, University of Victoria, V8P 5C2 Victoria, BC, Canada}
\email{bandersonsackaney@uvic.ca}

\author{Roland Vergnioux}
\address{Roland Vergnioux, Normandie Univ, UNICAEN, CNRS, LMNO, 14000 Caen, France} 
\email{roland.vergnioux@unicaen.fr}

\makeatletter
\@namedef{subjclassname@2020}{%
  \textup{2020} Mathematics Subject Classification}
\makeatother


\begin{document}

\begin{abstract}
    We introduce and investigate several quantum group dynamical notions for the purpose of studying $C^*$-simplicity of discrete quantum groups via the theory of boundary actions. In particular we define a quantum analogue of Powers' Averaging Property (PAP) and a quantum analogue of strongly faithful actions. We show that our quantum PAP implies $C^*$-simplicity and the uniqueness of $\sigma$-KMS states, and that the existence of a strongly $C^*$-faithful quantum boundary action also implies $C^*$-simplicity and, in the unimodular case, the quantum PAP. We illustrate these results in the case of the unitary free quantum groups $\F U_F$ by showing that they satisfy the quantum PAP and that they act strongly $C^*$-faithfully on their quantum Gromov boundary. Moreover we prove that this particular action of $\F U_F$ is a quantum boundary action.
\end{abstract}

\maketitle


\section{Introduction}

A discrete group $G$ is said to be $C^*$-simple if its reduced $C^*$-algebra $C^*_r(G)$ is simple, meaning that it has no non-trivial proper closed two-sided ideals. The theory of $C^*$-simplicity began with the work of Powers in \cite{P75} in which he established that the free group on two generators satisfies a certain strong group-theoretic averaging condition now known as Powers' averaging property (\PAP), and showed that this condition implies $C^*$-simplicity. This strong averaging property has been a prominent part of the theory since. Remarkably, it was later shown independently by Haagerup and Kennedy respectively in \cite{H16, K20} that Powers' averaging property is equivalent to $C^*$-simplicity. 

On the other hand, Kalantar and Kennedy discovered in \cite{KK14} a surprising connection between $C^*$-simplicity and the theory of boundary actions in topological group dynamics, which had been initiated by Furstenberg in the 1950s. Recall that an action of a group $G$ on a compact space is called a boundary action if it is minimal and strongly proximal. Kalantar and Kennedy established more precisely that a group $G$ is $C^*$-simple if and only if it acts freely on its Furstenberg boundary $\partial_F G$, if and only if it admits some topologically free boundary action. A key observation for these results is the fact, \cite[Remark 4]{H78}, that the algebra of continuous functions on the Furstenberg boundary is equivariantly isomorphic to Hamana's injective envelope of $\C$ in the category of $G$-equivariant operator systems. 

A bit later it was proved in \cite{BKKO17} that a group acts faithfully on its Furstenberg boundary if and only if it has the unique trace property, i.e., its reduced $C^*$-algebra has a unique trace. These results gave a solution to one direction of a long-standing open conjecture, namely that $C^*$-simplicity implies the unique trace property --- a counterexample to the converse was found in \cite{B17}. 

\bigskip

$C^*$-simplicity of a discrete {\em quantum} group $\G$ is defined exactly as in the classical case, using the reduced $C^*$-algebra $C^*_r(\G)$. The study of $C^*$-simplicity in this framework begins with Banica's work in \cite{B97} where he proved that Wang's free unitary discrete quantum groups $\F U_F$ \cite{Wang3,WangVanDaele} are $C^*$-simple by adapting Powers' methods to the quantum setting. The free orthogonal quantum groups $\F O_F$ and the quantum groups of quantum automorphisms of finite-dimensional $C^*$-algebras (equipped with their canonical trace) were later proved to be $C^*$-simple as well under some restrictions on the parameter matrix $F$, resp.\ the dimension of the considered $C^*$-algebra, see \cite{VV07, B13}.

The theory of boundary actions has also been extended to the setting of discrete quantum groups in \cite{KKSV22}, using the connection with Hamana's work mentioned above: the quantum Furstenberg boundary $C(\partial_F\G)$ is e.g.\ defined to be Hamana's $\G$-injective envelope of $\C$. Then, in \cite{KKSV22, ASK23} it was shown that a unimodular discrete quantum group has the unique trace property if and only if it acts faithfully on its Furstenberg boundary. Furthermore, in \cite{ASK23, deRo} it was shown that $C^*$-simplicity implies faithfullness of the action on the Furstenberg boundary. In particular, in the unimodular case, $C^*$-simplicity implies the unique trace property.

Despite these results, it seems that the theory of $C^*$-simplicity for quantum groups remains underdeveloped. For instance, an appropriate analogue of the notion of (topological) freeness for actions of discrete quantum groups on noncommutative $C^*$-algebras for the purposes of studying $C^*$-simplicity has yet to appear in the literature. A version of Powers' averaging property has not been developed for quantum groups either. In this paper we aim to address this gap in the theory as follows:
\begin{itemize}[nosep]
\item we propose quantum analogues of the {\PAP} and of topologically free boundary actions, 
\item we prove that these properties imply $C^*$-simplicity,
\item we prove that they are satisfied in our test case, namely the free unitary quantum groups $\F U_F$.
\end{itemize}

\bigskip

It is an important feature of our approach that we study these properties specially in relation to the ``rigidity'' of $\G$-equivariant ucp maps $C^*_r(\G)\to C(\partial_F\G)$. A deep result of Kennedy in \cite{K20}, used in its proof of the equivalence between $C^*$-simplicty and the \PAP, is that $C^*$-simplicity is characterized by uniqueness of a $G$-equivariant ucp map $C^*_r(G)\to C(\partial_F G)$. This result has had various generalizations, including a noncommutative analogue in \cite{Z19} where it was shown that a certain freeness property for an action of a group $G$ on a unital $C^*$-algebra $A$ is equivalent to the uniqueness of conditional expectations from the reduced crossed product $A\rtimes_rG$ onto $A$. Concerning a more general context for such problems, there has been interest in the ``rigidity'' (e.g.\ uniqueness) of (pseudo)-conditional expectations $A\to B$ of a $C^*$-inclusion $A\subseteq B$, cf.\ \cite{P12, PSZ23, PZ15, Z19}.

Our first observation is in Proposition~\ref{SimplicityCondition}, which states that $\G$ is $C^*$-simple if and only if every $\G$-equivariant ucp map $C^*_r(\G)\to C(\partial_F\G)$ is faithful. We consider also two additional properties: the uniqueness of $\G$-equivariant ucp maps $C^*_r(\G)\to C(\partial_F\G)$, and when all $\G$-equivariant ucp maps $C^*_r(\G)\to C(\partial_F\G)$ factor the canonical Haar state $h : C^*_r(\G)\to\C$ in the sense of Definition~\ref{Factor Haar State Def}. Note that we always have the canonical ucp map $a\mapsto h(a)1$ from $C^*_r(\G)$ to $C(\partial_F\G)$, which is faithful and factors $h$, but it is $\G$-equivariant if and only if $\G$ is unimodular --- this is one of the difficulties that are specific to the quantum setting.

\bigskip

We consider two versions of {\em Powers' averaging property} that generalize the {\PAP} for groups, namely the {\PAP} and {\PAPh} (see Definition~\ref{PAP Def}), where we note that {\PAPh} $\implies$ {\PAP}. Again in the unimodular case we have the simplification {\PAP} = {\PAPh}. We observe that the {\PAP} implies that $\G$ is $C^*$-simple and the {\PAPh} implies additionally that $C^*_r(\G)$ has a unique $\sigma$-KMS state (Corollary~\ref{PowersAveSimplicity}). Also, we observe that what Banica really showed in \cite{B97} is the {\PAPh} for the free unitary quantum groups $\F U_F$ (Propostion~\ref{Banicas Result}).

One of Kennedy's main results of \cite{K20} states that a group has Powers' averaging property if and only if the only $G$-boundary contained in $S(C^*_r(G))$ is trivial. In the quantum setting, $\G$-boundaries inside $S(C^*_r(\G))$ will be too small in general since they are classical. We, however, are able to obtain the following quantum analogue of Kennedy's result by replacing $\G$-boundaries with $\G_h$-boundary envelopes (see Definitions \ref{Envelopes Notation} and \ref{Boundary Envelopes Def}). \\[\medskipamount]
{\bf\noindent Theorem~\ref{GBoundariesPAP}.} {\em\nopagebreak A discrete quantum group $\G$ has the 
{\PAPh} if and only if the only 
$\G_h$-boundary envelope in $S(\Cred\G)$ is trivial.} \medskip

In Theorem~\ref{GBoundariesPAP} we also observe a stationary dynamical characterization of the \PAPh{} in the case where $C^*_r(\G)$ is separable.

Kennedy observed in \cite{K20} that the $G$-boundaries inside $S(C^*_r(G))$ are in bijection with $G$-equivariant ucp maps $C^*_r(G)\to C(\partial_F G)$. It is with this observation that he was able to characterize the Powers' averaging property via the uniqueness of $\G$-equivariant ucp maps $C^*_r(G)\to C(\partial_FG)$. We make similar observations in our work and, as an application of Theorem~\ref{GBoundariesPAP}, we obtain the following. \pagebreak[3] \\[\medskipamount]
{\bf\noindent Corollary~\ref{PAPUniqueGMap}.} \emph{\begin{enumerate}[nosep]
    \item A discrete quantum group $\G$ has the {\PAP} {\bf iff} every $\G$-equivariant ucp map $C^*_r(\G)\to C(\partial_F\G)$ factors the Haar state.
    \item If $\G$ is unimodular, then it has the {\PAP} {\bf iff} there is a unique $\G$-equivariant ucp map $C^*_r(\G)\to C(\partial_F\G)$.
\end{enumerate}}

\bigskip

Afterwards, we investigate a notion of {\em freeness} for quantum group actions which generalizes the notion of a group acting freely on a $C^*$-algebra in the sense of \cite{Z19} - note that, despite the nomenclature, a group acting freely on a commutative $C^*$-algebra is equivalent to the group acting topologically freely on the underlying topological space. In fact, we prove that an action $\G\acts A$ is free if and only if there exists a unique conditional expectation $A\rtimes_r\G\to A$ (Theorem~\ref{UniqueConditionalExpectations}), generalizing one of the main results of \cite{Z19}. We note, however, that this result implies that $\G\acts C(\partial_F\G)$ is never free when $\G$ is non-unimodular (see Corollary~\ref{TopFreenessImpliesPAPKac}). This observation suggests that this might not be the right notion to consider in the theory of $C^*$-simplicity. On the other hand, we apply the techniques used in the proof of Theorem~\ref{UniqueConditionalExpectations} to prove that faithfulness of an action $\G\acts A$ is equivalent to uniqueness of conditional expectations $A\rtimes_r\G\to A$ that restrict to states on $C^*_r(\G)$ (Theorem~\ref{Faithful Coaction Theorem}). 

In the search of a notion in quantum topological dynamics that would characterize, or at least imply, $C^*$-simplicity, while allowing for boundary actions of non-unimodular quantum groups, we introduce {\em strong $C^*$-faithfulness} for actions of discrete quantum groups on $C^*$-algebras (Definition~\ref{Def C*faithful}). In the classical case of discrete groups acting on locally compact spaces, it recovers strong faithfulness as considered e.g.\ in \cite[Lemma~4]{delaHarpe} and \cite[Section~2.1]{FLMMS}. Note that in this classical setting, a minimal group action on a compact space is strongly faithful if and only if it is topologically free, and recall that boundary actions are in particular compact and minimal, so that the main result of \cite{KK14} can be rephrased by saying that a discrete group $G$ is $C^*$-simple if and only if it admits a strongly faithful boundary action. In the quantum setting we are able to prove one direction of that result: \\[\medskipamount]
{\bf\noindent Theorem~\ref{Strong Faithfulness Implies Faithful Expectations} and Corollary~\ref{Strongly C* faithful Implies C*simple and Faithful}.} \emph{If $\G$ admits a $\G$-boundary with strongly $C^*$-faithful action, then $\G$ has the \PAP{}. In particular, $\G$ is $C^*$-simple.}

\bigskip

We finally come back to our test example $\F U_F$ with the aim of showing that it satisfies the hypothesis of the above theorem, with respect to the quantum Gromov boundary $A = C(\partial_G\F U_F)$ introduced in \cite{VVV10}. We first show that this action is strongly $C^*$-faithful by using a combinatorial trick uncovered by Banica in \cite{B97}. On the other hand, the fact that $C(\partial_G\F U_F)$ is a $\G$-boundary has been established in \cite{HHN22} but only with respect to the natural action of the Drinfeld double $\G = D(\F U_F)$. We upgrade this result by showing that $C(\partial_G\F U_F)$ is in fact already an $\F U_F$-boundary --- this is much more intricate and relies on the unique stationarity method used in \cite{KKSV22} in the case of $\F O_F$. We prove more precisely the following.\\[\medskipamount]
{\bf\noindent Theorem~\ref{Unique stationary state}.} {\em Let $F\in GL_N(\C)$ with $N\geq 3$.
Then $C(\partial_G\F U_F)$ has a unique stationary state with respect to the ``nearest neighboor'' quantum random walk on $\F U_F$.
}\medskip

As an intermediate step for the non-unimodular case we establish a non-atomicity result for the restriction of stationary states to the {\em classical} Gromov boundary of $\F U_F$. Taking into account the results of \cite{KKSV22} and \cite{VVV10}, the above theorem implies that $C(\partial_G \F U_F)$ is indeed an $\F U_F$-boundary, and the strong $C^*$-faithfulness property together with our general results yield a new ``dynamical proof'' of $C^*$-simplicity of $\F U_F$ when $F \in GL_N(\C)$, $N\geq 3$.

\bigskip

We complete the introduction by presenting the layout of this paper. In Section $2$ we discuss the preliminaries on discrete quantum groups, boundary actions, free unitary discrete quantum groups, and $C^*$-simplicity. In Section $3$ we develop notions of Powers' averaging property for discrete quantum groups. Here, we prove a characterization in terms of ``boundary envelopes'' in the state space of $C^*_r(\G)$ and, in the unimodular case, a characterization in terms of $\G$-equivariant ucp maps $C^*_r(\G)\to C(\partial_F\G)$ as discussed above. In Section $4$ we develop a notion of a free action and a notion of a strong $C^*$-faithful action, and investigate the connections with $C^*$-simplicity and the \PAP. Finally, in Section $5$ we prove that the Gromov boundary of $\F U_F$ is a $\F U_F$-boundary when $N\geq 3$.

\section{Preliminaries}

\subsection{Discrete quantum groups and their boundaries}

Discrete quantum groups were introduced as duals of compact quantum groups in \cite{PodlesWoronowicz_Lorentz} ; an operator-theoretic characterization is given in \cite{BaajSkandalis_MultUnitaries} at the level of the associated multiplicative unitary, and algebraic characterizations are given in \cite{EffrosRuan_DQG} and \cite{VanDaele_DQG} at the level of the associated algebras. We shall mostly follow the notation and conventions established in \cite{KV00} for general locally compact quantum groups. 

So let $\G$ be a locally compact quantum group, given by the reduced Hopf-$C^*$-algebra $c_0(\G)$. We choose a GNS space $(H,\Lambda)$ for the left Haar weight $h_L$ of $\G$ and construct the left multiplicative unitary $W\in B(H\otimes H)$ of $\G$, given by the formula $W^*(\Lambda(f)\otimes\Lambda(g)) = (\Lambda\otimes\Lambda)(\Delta(g)(f\otimes 1))$ for $f$, $g\in c_c(\G)$. One can then recover $c_0(\G)$ as the norm closure of $\{(\id\otimes\omega)(W) \mid \omega\in B(H)_*\}$, equipped with the coproduct $\Delta(f) = W^*(1\otimes f)W$. Throughout the article we assume $\G$ and $W$ to be of discrete type, i.e.\ there exists a unit vector $\eta\in H$ such that $W(\eta\otimes\zeta)=\eta\otimes\zeta$ for all $\zeta\in H$. The $C^*$-algebra $c_0(\G)$ then contains a distinguished dense multiplier Hopf $*$-algebra \cite{VanDaele_Multiplier} that we denote $c_c(\G)$.

In this article, we are interested in the structure of the reduced group
$C^*$-algebra $\Cred\G$ of $\G$, defined as the norm closure of
$\{(\omega\otimes\id)(W) \mid \omega\in B(H)_*\}$ in $B(H)$. We have then
$W \in M(c_0(\G)\otimes\Cred\G)$ and we endow $\Cred\G$ with the adjoint action
of $\G$ given by the coaction $\ad : \Cred\G\to M(c_0(\G)\otimes\Cred\G), x\mapsto W^*(1\otimes x)W$. 
Sometimes it will be more convenient to use the adjoint action of the
opposite discrete quantum group $\G^\op$ given by the $*$-homomorphism
$\ad^\op : \Cred\G\to M(c_0(\G)\otimes\Cred\G), x\mapsto W(1\otimes x)W^*$,
which is indeed a coaction for the opposite coproduct
$\Delta^\op = \sigma\circ\Delta$ on $c_0(\G^\op) = c_0(\G)$.
Furthermore, the adjoint actions of $\G$ and $\G^\op$ are related by the equation $\ad^\op(x) = (R\otimes \hat{R})\ad(\hat{R}(x))$, where $R$ and $\hat R$ are the unitary antipodes of $c_0(\G)$ and $\Cred\G$ respectively, satisfying the identity $(R\otimes\hat R)(W) = W$.

\bigskip

We denote $\Corep(\G)$ the category of non-degenerate finite-dimensional
$*$-representations of $c_0(\G)$, and $\Irr(\G) = I$ the set of irreducible objects up to
equivalence. Note that such representations
$\pi : c_0(\G) \to B(H_\pi)$ correspond to unitary elements
$w = (\pi\otimes\id)(W) \in B(H_w)\otimes\Cred\G$, where $H_w=H_\pi$, and we
will in fact rather use this second picture.  The category $\Corep(\G)$ has a natural monoidal structure given by
$\pi\otimes\rho := (\pi\otimes\rho)\circ\Delta$ and
$v\otimes w := v_{13}w_{23}$. Since $\G$ is discrete, we have an
isomorphism $c_0(\G) \simeq \bigoplus^{c_0}_{w\in I} B(H_w)$ such that
$W = \bigoplus_{w\in I}w$, and we denote $\ell^\infty(\G) = M(c_0(\G))$ the corresponding $\ell^\infty$-direct sum.
We also denote $p_w\in c_0(\G)$ the minimal central projection corresponding to $\id\in B(H_w)$ in this isomorphism.

The monoidal category $\Corep(\G)$ is rigid, in particular we can find for each
$v\in\Corep(\G)$ another corepresentation $\bar v$, unique up to isomorphism,
and morphisms $t_v : \C\to H_v\otimes H_{\bar v}$,
$s_v : \C\to H_{\bar v}\otimes H_v$ satisfying the conjugate equations and
normalized so that $t_v^*t_v = s_v^* s_v =: \qdim(v)$. If $v=w$ is irreducible,
these morphisms are unique up to a phase and we consider the associated left and
right quantum traces, $\qTr_w(a) = t_w^*(a\otimes 1)t_w$ and
$\qTr'_w(a) = s_w^*(1\otimes a)s_w$, for $a\in B(H_w)$, as well as the
corresponding states $\qtr_w = (\qdim w)^{-1}\qTr_w$,
$\qtr'_w = (\qdim w)^{-1}\qTr'_w$. We can also write $\qTr_w(a) = \Tr(Q_wa)$,
$\qTr'_w(a) = \Tr(Q_w^{-1}a)$ for a unique positive matrix $Q_w \in
B(H_w)$. Compare \cite[Notation~1.11]{VV07}, where the vectors $t_w$, $s_w$ are
however normalized differently. Note that the conjugate equations yield $(\id\otimes\qTr_w)(t_wt_w^*) = \id_w$, whereas we have $(\qTr_w\otimes\id)(t_wt_w^*) = Q_w^{-2}$.

Using this data we can write down explicitly the left- and right-invariant Haar
weights on $c_0(\G)$:
\begin{displaymath}
  h_L(a) = \sum \qdim(w)^2 \qtr_w(a_w) \quad\text{and}\quad
  h_R(a) = \sum \qdim(w)^2 \qtr'_w(a_w),
\end{displaymath}
for $a = (a_w)_w \in c_c(\G)$. Compare \cite[Proposition~1.12]{VV07}.  Note that
since we have $h_L(ab) = h_L(bS^2(a))$ in any discrete quantum group algebra,
the above formula implies $S^2(a) = Q_waQ_w^{-1}$ for
$a\in B(H_w) \subset c_c(\G)$, and we have $S(Q_w)=Q_w^{-1}=S^{-1}(Q_w)$. Recall also that if we switch from $\G$ to
$\G^\op$ we have to exchange $h_L$ with $h_R$, and to replace $S$ with $S^{-1}$ (but the unitary antipode $R = R^{-1}$ is unchanged).

\bigskip

There are two ways to endow $\Cred\G$ with the structure of a Woronowicz
$C^*$-algebra, we choose the coproduct
$\Delta : \Cred\G\to \Cred\G\otimes \Cred\G$ such that
$(\id\otimes\Delta)(W) = W_{12}W_{13}$, as in \cite{VV07}. We denote $\hat\G$
the compact quantum group given by $C^r(\hat\G) = \Cred\G$ and $\Delta$. Then
the unitaries $W$ and $w\in I$ above are corepresentations of $(\Cred\G,\Delta)$
in the sense of \cite{W87} and $\Corep(\G)$ identifies with the category of
finite-dimensional unitary representations of $\hat\G$. We denote
$\C[\G] = \fO(\hat\G)$ the canonical dense Hopf-$*$-subalgebra spanned by
coefficients of finite-dimensional corepresentations.

We denote $h \in\Cred\G^*$ the Haar state, and we can write down the
Woronowicz-Peter-Weyl orthogonality relations using the matricial states $\qtr$,
$\qtr'$ introduced above:
\begin{displaymath}
  (\id\otimes h)(w(a\otimes 1)w^*) = \qtr_w(a), \quad
  (\id\otimes h)(w^*(a\otimes 1)w) = \qtr'_w(a),
\end{displaymath}
for $a\in B(H_w)$, $w\in I$. See \cite{VV07}, Notation~1.11. We recognize the
modular matrices $F_w = Q_w$ from \cite{W87} for $(\Cred\G,\Delta)$, which are connected to
Woronowicz' characters $f_z\in\C[\G]^*$ by the formula
$Q_w^z = (\id\otimes f_z)(w)$.

Recall that the modular group of $h$ is implemented by these characters, more
precisely we have $h(xy) = h(y\sigma_{-i}(x))$ if we put
$\sigma_{z}(x) = f_{iz}*x*f_{iz}$ for $x$, $y\in\C[\G]$. Here we are using the
convolution products $\varphi*x = (\id\otimes\varphi)\Delta(x)$,
$x*\varphi = (\varphi\otimes\id)\Delta(x)$ for $x\in\C[\G]$,
$\varphi \in \C[\G]^*$. This yields
$(\id\otimes h)((1\otimes y)w) = (\id\otimes h)((Q_w^{-1}\otimes
1)w(Q_w^{-1}\otimes 1)(1\otimes y))$ for $w\in I$. Since on the other hand
$\qtr_w(ba) = \qtr_w(Q_w^{-1}aQ_wb)$, we obtain
\begin{displaymath}
  (\qtr_w\otimes h)(w^*(1\otimes y)w) = (\qtr_w\otimes h)((Q_w^{-2}\otimes 1)ww^*(1\otimes y))
  = \qtr'_w(1)h(y) = h(y).
\end{displaymath}
In terms of the adjoint action $\ad$ this can be written $(\qtr_w\otimes h) \ad = h$. Similarly one can check that $(\qtr'_w\otimes h)\ad^\op = h$.

\bigskip

An action of a discrete quantum group $\G$ on a $C^*$-algebra $A$ is given by a non-degenerate $*$-homo\-morphism $\alpha : A\to M(c_0(\G)\otimes A)$ such that
$(\id\otimes\alpha)\alpha = (\Delta\otimes\id)\alpha$ and $\alpha(A)(c_0(\G)\otimes 1)$ is a dense subspace of $c_0(\G)\otimes A$. We will also say that $A$ is
a $\G$-$C^*$-algebra. For $w\in I$ we denote $\alpha_w = (p_w\otimes 1)\alpha : A \to M(B(H_w)\otimes A)$.
A completely positive (cp) map $\Phi : A \to B$ between $\G$-$C^*$-algebras $(A,\alpha)$, $(B,\beta)$ is called $\G$-equivariant if
$(\id\otimes\Phi)\alpha = \beta\Phi$. For $f\in \ell^1(\G) = c_0(\G)^*$ and $a\in A$ we denote $f*a = (f\otimes\id)\alpha(a) \in M(A)$. For $\nu\in A^*$ we
consider the Poisson transform $\fP_\nu = (\id\otimes\nu)\alpha: A\to M(c_0(\G))$. Throughout we will assume that $A$ is unital unless otherwise specified.

The reduced crossed product $A\rtimes_r\G$ is the closed subspace of $M(K(H)\otimes A)$ generated by the elements $\alpha(a)(x\otimes 1)$ with $a\in A$,
$x\in\Cred\G$. It is equipped with an action of the dual of $\G$ and, more importantly in this paper, by the action of $\G$ associated to the map
$\ad_A : A\rtimes_r\G\to M(c_0(\G)\otimes (A\rtimes_r\G))$, $X \mapsto (W^*\otimes 1)(1\otimes X)(W\otimes 1)$. The non-degenerate injective $*$-homomorphisms
$\alpha : A \to A\rtimes_r\G$ and $x\in\Cred\G\to x\otimes 1\in A\rtimes_r\G$ are then equivariant with respect to $\ad_A$, $\alpha$ and $\ad$. If there is no
risk of confusion we will denote $\ad_A = \ad$.

We say that a unital $\G$-$C^*$-algebra $A$ is a $\G$-boundary if for every state $\nu\in S(A)$ the Poisson map $\fP_\nu$ is completely isometric (ci). Equivalently,
every unital cp (ucp) equivariant map $\Phi : A\to B$ to any unital $\G$-$C^*$-algebra $B$ is completely isometric, cf \cite[Section~4]{KKSV22}. There exists a
$\G$-boundary in which every $\G$-boundary embeds equivariantly and completely isometrically, and it is unique up to (unique) equivariant $*$-isomorphism. It is
called the Furstenberg boundary of $\G$, denoted $C(\partial_F\G)$ \cite[Theorem~4.16]{KKSV22}. Moreover, this $\G$-boundary is $\G$-injective, meaning that for
any equivariant uci map $\Phi : A \to B$, every equivariant ucp map $\psi : A \to C(\partial_F\G)$ extends to $\Psi : B \to C(\partial_F\G)$.

\bigskip

As in the classical case, a discrete quantum group $\G$ is called {\bf $C^*$-simple} if $\Cred\G$ is a simple $C^*$-algebra. In the classical case it is known
that there are enough ucp maps $\Psi : \Cred\G \to B$ to detect $C^*$-simplicity --- see, for example, \cite[Proposition 3.1]{KS22} and for a generalization to
(noncommutative) crossed products (of groups) see \cite[Theorem 6.6]{KS19}. This fact holds in the setting of DQGs as well.

More precisely, for a $\G$-equivariant ucp map $\Psi : A\to B$ we denote $I_\Psi = \{a\in A : \Psi(a^*a) = 0\}$, which is automatically a closed left ideal by
the Schwarz inequality for ucp maps. Recall that $\Psi$ is called faithful if $I_\Psi = \{0\}$.
\begin{prop}\label{SimplicityCondition}
  Let $\G$ be a DQG and $(A,\alpha)$ a $\G$-$C^*$-algebra. For every
  $\G$-equivariant ucp map $\Psi : \Cred \G\to A$, $I_\Psi$ is a two-sided
  ideal. Moreover, for every closed two-sided ideal $I \subsetneq \Cred\G$ there
  exists a $\G$-equivariant ucp map $\Phi : \Cred \G\to C(\partial_F\G)$ such
  that $I\subseteq I_\Phi$. In particular, $\G$ is $C^*$-simple if and only if
  every $\G$-equivariant ucp map $\Cred\G \to C(\partial_F\G)$ is faithful.
\end{prop}

\begin{proof}
  Take $u\in\Irr(\G)$ and unit vectors $\zeta$, $\xi \in H_u$. For $x\in I_\Psi$ we can write
  \begin{align*}
    0 &= (\omega_\zeta\otimes\id)\alpha(\Psi(x^*x)) =
    (\omega_\zeta\otimes\Psi)(\ad(x^*x)) =
    (\omega_\zeta\otimes\Psi)(u^*(1\otimes x^*x)u)) \\ & \geq
    (\omega_\zeta\otimes\Psi)(u^*(\xi\xi^*\otimes x^*x)u)) =
    \Psi(u_{\zeta,\xi}^*x^*xu_{\zeta,\xi}).
  \end{align*}
  Hence $xu_{\zeta,\xi} \in I_\Psi$ and since the coefficients
  $u_{\zeta,\xi} = (\xi^*\otimes 1)u(\zeta\otimes 1)$ span a dense subspace of
  $\Cred\G$ we conclude that $I_\Psi$ is a right ideal.
    
  If $I\subset \Cred\G$ is a bilateral ideal, we have
  $\ad(I)(c_0(\G)\otimes\Cred\G) = W^*(1\otimes I)W(c_0(\G)\otimes\Cred\G)
  \subset c_0(G)\otimes I$, since $W\in M(c_0(\G)\otimes \Cred\G)$. Denoting
  $q : \Cred\G\to A = \Cred\G/I$ the quotient map, this shows that
  $(\id\otimes q)\ad$ factors to a coaction $\alpha$ on $A$ such that $q$
  is equivariant. Applying $\G$-injectivity of
  $C(\partial_F\G)$ to the unital inclusion $\C\subset A$ and the canonical unital map $\C\to C(\partial_F\G)$ we obtain a $\G$-equivariant ucp map  $\theta : A\to C(\partial_F\G)$. 
  Then $\Phi = \theta\circ q$ is a
  $\G$-equivariant ucp map such that $I\subseteq I_\Phi$.
\end{proof}

\subsection{The free unitary quantum groups and their boundary}
\label{Section Boundary FUF}

Most examples of discrete quantum groups are in fact defined as duals of compact matrix quantum groups. For instance, fix an integer $N \geq 2$ and
$F\in \GL_N(\C)$. The universal unitary compact quantum group $U^+_F$ is the compact quantum group given by the universal unital $C^*$-algebra $C^u(U_F^+)$
generated by the entries of unitary matrix $u\in M_N(C^u(U^+_F))$ and the relations making $F\bar u F^{-1}$ a unitary matrix as well, endowed with the coproduct such that
$\Delta(u_{ij}) = \sum u_{ik}\otimes u_{kj}$ \cite{Wang3,WangVanDaele}. Identifying $M_N(C^u(U^+_F))$ with $B(\C^N)\otimes C^u(U^+_F)$, the matrix $u$ becomes a
representation of $U_F^+$. 

The Woronowicz $C^*$-algebra $C^u(U^+_F)$ has a reduced version $C^r(U_F^+)$, and up to isomorphism there is a unique discrete quantum group $\F U_F$, given by a Hopf-$C^*$-algebra $(c_0(\F U_F),\Delta)$, such that following the constructions explained previously we have $(C^*_r(\F U_F), \Delta) \simeq (C^r(U_F^+),\Delta)$. We will accordingly denote $C^*(\F U_F) = C^u(U^+_F)$, so that we can write $u \in M_N(\C)\otimes C^*(\F U_F)$. Note that $\F U_F$ is unimodular if and only if $F$ is a scalar multiple of a unitary matrix.

\bigskip

The category of corepresentations of $\F U_F$ has been computed by Banica \cite{B97}. The elements of $I = \Irr(\G)$ can be
labelled by words $x$ on the letters $u$, $\bar u$ in such a way that the empty word $1$ corresponds to the trivial corepresentation, the one letter words $u$
and $\bar u$ correspond to the fundamental corepresentation $(u_{ij})$ and its unitary dual $F(u_{ij}^*)F^{-1}$, and we have the recursive fusion rules
\begin{equation}\label{eq_fusion_rules}
\begin{gathered} 
  xu\otimes uy \simeq xuuy, \quad xu\otimes \bar u y \simeq xu\bar uy \oplus x\otimes y, \\
  x\bar u\otimes \bar uy \simeq x\bar u\bar uy, \quad x\bar u\otimes u y \simeq
  x\bar uuy \oplus x\otimes y.
\end{gathered}    
\end{equation}
In particular $w$ is a subobject of $x\otimes y$ {\bf iff} we can write $x=x'v$, $y=\bar v y'$ and $w=x'y'$, and for $x$,
$y\in I \setminus \{1\}$ we have $xy \simeq x\otimes y$ {\bf iff} the last
letter of $x$ equals the first of $y$.
Note that we identify a word in $u$, $\bar u$ and a representant for the
corresponding equivalence class of irreducible corepresentations.  We denote $p_x\in\ell^\infty(\G)$ the
minimal central projection corresponding to $x$, and this yields an
identification $Z(\ell^\infty(\G)) = \ell^\infty(I)$. The length of $x\in I$ as
a word on $u$, $\bar u$ is denoted $|x|$ and we put
$I_n = \{x \in I \mid |x| = n\}$,
$p_n = \sum_{|x| = n} p_x \in \ell^\infty(\G)$. We denote $z\geq x$ if $z = xy$
for some $y\in I$, and we put $I(x) = \{z\in I \mid z\geq x\}$.

It follows from the fusion rules that inclusions of irreducibles
$x\subset y\otimes z$ are always multiplicity free and we choose corresponding
isometric intertwiners $V(x,y\otimes z)$, which are unique up to a phase. We
also denote $P(x,y\otimes z) = V(x,y\otimes z)V(x,y\otimes z)^* \in B(H_y\otimes H_z)$ the corresponding range projections, and we put $P(x,y\otimes z) = 0$ if $x\not\subset y\otimes z$. These intertwiners can be used to compute the coproduct of $c_0(\G)$: for $a\in p_x c_0(\G) \simeq B(H_x)$ and $y$, $z\in I$ we have
\begin{displaymath}
    (p_y\otimes p_z)\Delta(a) = V(x,y\otimes z) a V(x,y\otimes z)^*.
\end{displaymath}

Let $q$ be the unique number in $]0,1]$ such that
$q+q^{-1} = \qdim(u) = \qdim(\bar u)$. Note that $q = 1$ if and only if $N=2$
and $\F U_F$ is unimodular, i.e.\ $Q_u = I_2 = Q_{\bar u}$. For a letter $\alpha = u$ or $\bar u$, denote
$\alpha^{(k)} = \alpha \bar\alpha \alpha \ldots$ ($k$ terms) and
$\alpha^{(\infty)}$ the infinite alternating word starting with $\alpha$. We
have
$\bar\alpha \otimes \alpha^{(k)} \simeq
\bar\alpha^{(k+1)}\oplus\bar\alpha^{(k-1)}$ for $k\geq 1$ and it follows easily
that $\qdim(\alpha^{(k)}) = [k+1]_q$, using the $q$-numbers
$[n]_q = (q^{-n}-q^n)/(q^{-1}-q)$. Then, decomposing $x\in I$ as
$x = x_1\otimes\cdots\otimes x_p$ where each $x_i$ is of the form
$\alpha^{(k)}$, we obtain $\qdim(x) = [|x_1|+1]_q\cdots[|x_p|+1]_q$.

We will need the following lemma about Woronowicz' modular matrices $Q_x$.

\begin{lem} \label{lem_woro_matrices}
  Denote $\rho = \max(\|Q_u\|, \|Q_{\bar u}\|)$. For all $x\in I$ we have $\|Q_x\|/\qdim(x) \leq (q\rho)^{|x|}$ and $\|Q_x^{-1}\|/\qdim(x) \leq (q\rho)^{|x|}$. Moreover if $N\geq 3$ we have $q\rho < 1$.
\end{lem}

\begin{proof}  
  If $x = \alpha_1\cdots\alpha_n$, with $\alpha_i \in \{u,\bar u\}$, $n = |x|$, we have in particular $x\subset v := \alpha_1\otimes\cdots\otimes\alpha_n$,
  hence $Q_x$ appears as a diagonal block of $Q_{\alpha_1}\otimes\cdots\otimes Q_{\alpha_n}$ via the decomposition of $v$ into irreducibles. As a result
  $\|Q_x\|\leq \prod_1^n \|Q_{\alpha_i}\| \leq \rho^n$. The same reasoning holds for $Q_x^{-1}$ with the $Q_{\alpha_i}^{-1}$'s, and since $Q_{\bar u} = \bar Q_u^{-1}$ we get the same upper bound.
  On the other hand we have $[k+1]_q = q^{-k} (1-q^{2k+2})/(1-q^2) \geq q^{-k}$ for all $k$, hence
  $\qdim(x) \geq q^{-|x|}$ for all $x\in I$. In particular $\|Q_x\|/\qdim(x) \leq (q\rho)^{|x|}$.
  
  Since $\Tr(Q_u) = \Tr(Q_u^{-1})$ we must have $\|Q_u\|$, $\|Q_u^{-1}\| \geq 1$. Assume $N\geq 3$. If $\|Q_u\| \geq \|Q_u^{-1}\|$ we have
  $q+q^{-1} = \Tr(Q_u) > \|Q_u\|+\|Q_u^{-1}\|^{-1} \geq \|Q_u\|+\|Q_u\|^{-1}$, since $Q_u$ has at least a third eigenvalue beyond $\|Q_u\|$ and
  $\|Q_u^{-1}\|^{-1}$. The same holds if $\|Q_u\| \leq \|Q_u^{-1}\|$ by writing instead
  $q+q^{-1} = \Tr(Q_u^{-1}) > \|Q_u^{-1}\|+\|Q_u\|^{-1} \geq \|Q_u\|+\|Q_u\|^{-1}$. Since $t\mapsto t+t^{-1}$ is strictly increasing on $[1,+\infty[$ this
  implies $q^{-1}>\|Q_u\|$. Since we also have $q+q^{-1} = \qdim(\bar u) = \Tr(Q_{\bar u}) = \Tr(Q_{\bar u}^{-1})$, we have similarly $q^{-1}>\|Q_{\bar u}\|$,
  hence $q^{-1}>\rho$.
\end{proof}

\bigskip

The construction of the Gromov compactification $\beta_G\F U_F$ of $\F U_F$
relies on the following ucp maps, defined for $x$, $y\in \Irr(\G)$ and
$m\in\N$:
\begin{align*}
  &\psi_{x,xy} : B(H_x)\to B(H_{xy}), ~ a\mapsto V(xy,x\otimes y)^*(a\otimes{\id}_y)V(xy,x\otimes y), \\
  &\psi_{l,m} = \sum\nolimits_{\substack{|x|=l \\ |y|=m-l}} \psi_{x,xy} : p_l \ell^\infty(\G) \to p_m\ell^\infty(\G) \quad \text{for $m\geq l$,} \\
  &\psi_{x,\infty} =  \sum_{y\in I} \psi_{x,xy} : B(H_x) \to \ell^\infty(\G), \quad
    \psi_{m,\infty} = \sum_{|x| = m}  \psi_{x,\infty} : p_m\ell^\infty(\G) \to \ell^\infty(\G).
\end{align*}
These maps define an inductive system, in the sense that $\psi_{xy,xyz}\circ \psi_{x,xy} = \psi_{x,xyz}$. The quantum Gromov compactification $\beta_G\F U_F$ is then given by the unital $C^*$-subalgebra $B = C(\beta_G \F U_F) = \bar B_0 \subset \ell^\infty(\G)$, which is the closure of
\begin{displaymath}
  B_0 = \{\psi_{m,\infty}(a) ; m\in\N, a\in p_m\ell^\infty(\G)\}.
\end{displaymath}
Finally the quantum Gromov boundary of $\F U_F$ is given by $C(\partial_G\F U_F) = B_\infty = B/c_0(\G)$, see \cite{VV07,VVV10}.

Note that as a free monoid the set $I$ has a natural Cayley tree structure (with respect to the generating set $\{u,\bar u\}$), and recall that $Z(\ell^\infty(\F U_F)) \simeq \ell^\infty(I)$. The restriction of $\psi_{x,xy}$ to the center is the canonical map $\id_x\mapsto\id_{xy}$ and it then follows
that $\ell^\infty(I)\cap C(\beta_G \F U_F)  = C(\beta I)$, where $\beta I$ is the usual compactification of the tree $I$ and $C(\beta I)$ is identified with a subalgebra of $\ell^\infty(I)$ via the restriction map. The boundary $\partial I = \beta I\setminus I$ is canonically identified with the set of infinite words in $u$, $\bar u$. Then $C(\partial_G\F U_F)$ is a continuous field of unital $C^*$-algebras over $\partial I$ \cite[Lemma~3.3]{VVV10}.

We will make important use of the central projections
$\pi_x = \psi_{x,\infty}(p_x)\in C(\beta_G \F U_F)$, for $x\in I$. More precisely we have $\pi_x = \sum_{y\in I} p_{xy}$, since
$V(xy,x\otimes y)^*V(xy,x\otimes y) = \id_{xy}$. This is a projection in
$\ell^\infty(I)\cap C(\beta_G \F U_F)$, whose image in $C(\partial I)$ corresponds to the subset
$\partial I(x) \subset \partial I$ of infinite words starting with $x$. Any
state $\nu$ on $C(\partial_G\F U_F)$ induces by restriction a probability measure on
$\partial I$ that we denote $\nu_I$. We have then
$\nu_I(\partial I(x)) = \nu(\pi_x)$ by definition.

The action $\beta : B \to M(c_0(\G)\otimes B)$ of $\G$ on $B$ is just given by
the restriction of the comultiplication. We have in particular
$(p_z\otimes p_t)\beta(\pi_x) = \sum_{y\in I} P(xy,z\otimes t)$, which is
non-zero iff $z\otimes t$ contains an irreducible corepresentation starting with
$x$. We will need the following description of the boundary action, which is implicit in \cite{VVV10}:

\begin{lem}\label{lem_boundary_action}
  Fix $k$, $n\in\N$ and $\epsilon > 0$. Denote $p_{\geq r} = \sum_{l\geq r} p_l$. Then there exists $r\geq k+n$ such that for all $a_k\in p_k \ell^\infty(\F U_F)$, $\|a_k\|\leq 1$ we have
  \begin{displaymath}
    \|(p_n\otimes p_{\geq r})\Delta(\psi_{k,\infty}(a_k)) - (p_n\otimes \psi_{r,\infty})(b)\|\leq\epsilon,
  \end{displaymath}
  where $b = \sum_{l=0}^n b_{n+r-2l}$,
  $b_s = (p_n\otimes p_r)\Delta(\psi_{k,s}(a_k))$.
\end{lem}

\begin{proof}
  The element $a_k$ can be decomposed into a sum of $2^k$ elements
  $a_x \in B(H_x)$ with $x\in I_k$. Up to replacing $\epsilon$ by
  $2^{-k}\epsilon$ and using the triangle inequality we can assume $a_k = a_x$
  for some $x \in I_k$.  Denote $\psi_{x,\infty}(a_x) = a$, so that $a_t = \psi_{x,t}(a_x)$ if $t\geq x$ and $a_t = 0$ else --- in particular $a_x = p_x a$ and $a_k = p_k a$ as expected. Take
  $r \geq k+n$.  We first consider $(p_y\otimes p_z)\Delta(a)$ with
  $|y| = n$, $z=z_1z_2$, $|z_1| = r$. Since $r\geq n$, subobjects of $y\otimes z$ must be of the form $tz_2$ with $t\subset y\otimes z_1$, hence we have
  \begin{displaymath}
    (p_y\otimes p_z)\Delta(a) = \sum_{tz_2\subset y\otimes z_1z_2}
    V(tz_2,y\otimes z_1z_2)a_{tz_2}V(tz_2,y\otimes z_1z_2)^*.
  \end{displaymath}
  Since $|t| \geq |z_1|-|y| \geq |x|$ we have $a_{tz_2} = \psi_{t,tz_2}(a_t)$ 
  --- both terms vanish unless $t\geq x$.
  Recall from \cite[Lemma~7.8.3, equation~(8.47)]{VVthesis} that we have
  ${V(tz_2,y\otimes z_1z_2)}{V(tz_2,t\otimes z_2)^*} \simeq \mu (\id_y\otimes
  {V(z_1z_2,z_1\otimes z_2)^*})$ $({V(t,y\otimes z_1)}\otimes\id_{z_2})$ up to
  $C q^{(|t|+|z_1|-|y|)/2} \leq C q^{(r-n)/2}$ in operator norm, for some $\mu\in\C$ such that $|\mu|=1$. This yields
  \begin{align*}
    &V(tz_2,y\otimes z_1z_2)a_{tz_2}V(tz_2,y\otimes z_1z_2)^* = \\
    &\hspace{3cm} = V(tz_2,y\otimes z_1z_2)V(tz_2,t\otimes z_2)^*(a_t\otimes\id_{z_2})
      V(tz_2,t\otimes z_2)V(tz_2,y\otimes z_1z_2)^* \\
    &\hspace{3cm} \simeq (\id_y\otimes\psi_{z_1,z_1z_2})(V(t,y\otimes z_1)a_tV(t,y\otimes z_1)^*) = (p_y\otimes p_z\psi_{r,|z|})\Delta(a_t)
  \end{align*}
  up to $2Cq^{(r-n)/2}$. Putting $s = |t|$, $t$ is the only subobject of $y\otimes z_1$ of length $s$, so that
  $(p_y\otimes p_{z_1})\Delta(a_t) = (p_y\otimes p_{z_1})\Delta(p_s\psi_{x,\infty}(a_x)) = (p_y\otimes p_{z_1})b_s$. Summing over $t$ we obtain
  \begin{displaymath}
    \|(p_y\otimes p_z)\Delta(a) - (p_y\otimes p_z\psi_{r,|z|})(b)\| \leq 2C(n+1)q^{(r-n)/2},
  \end{displaymath}
  since there are at most $n+1$ possible values of $s$. For $r$ large enough this yields the result, by taking the supremum over $y$ and $z$.
\end{proof}

\section{The Powers Averaging Property}

\subsection{Definition and Basic Results}
\label{sec_def_PAP}

Powers' averaging property is a combinatorial technique originally used by
Powers in \cite{P75} to prove that free groups are $C^*$-simple with the unique
trace property. Later, it was shown independently by Kennedy \cite{K20} and
Haagerup \cite{H16} that the converse holds, namely that every $C^*$-simple
group has Powers' averaging property. We introduce below an analogue of Powers' 
averaging property for discrete quantum groups and show that it still implies
$C^*$-simplicity.

Denote $\Prob(\G)$ the convex space of normal states on $\ell^\infty(\G)$ and
$\Probc(\G)$ the subspace of states with finite support, i.e.\ states $f$ for
which there exists a projection $p\in c_c(\G)$ with $f(p)=1$.  We denote
$h\in S(\Cred\G)$ the Haar state and we consider the following convolution
operations, for $x\in\Cred\G$, $f\in\Prob(\G)$, $\mu\in S(\Cred\G)$:
$x*f = (f\otimes \id)\ad(x)$, $f*\mu = (f\otimes \mu)\circ\ad$,
$\mu*x = (\id\otimes \mu)\Delta(x)$, $x*\mu = (\mu\otimes\id)\Delta(x)$. The
state $\mu$ on $\Cred\G$ is called $\G$-invariant if we have $f*\mu = \mu$ for
all $f\in\Prob(\G)$. The Haar state $h$ is $\G$-invariant in the classical case,
but not always in the general quantum framework, and this motivates the following definition.

\begin{defn}\label{Probh Def}
  We define $\Probh(\G) = \{f\in \Prob(\G) : f*h = h\}$ and
  $\Probhc(\G) = \Probh(\G) \cap \Probc(\G)$.
\end{defn}

From the preliminaries we have important examples of states in $\Probhc(\G)$, namely the quantum traces $\qtr_w$. Before defining the \PAP{} let us note the following stronger stabilization property for elements of $\Probhc(\G)$.

\begin{lem}\label{KMS States and h-Invariant States}
  Let $\G$ be a DQG. If $f\in \Probhc(\G)$ then $f*\tau = \tau$ for every
  $\sigma$-KMS state $\tau\in S(\Cred\G)$.
\end{lem}

\begin{proof}
  Note that we have $f\in\Probh(\G)$ {\bf iff}
  $h(x) = (f\otimes h)(W^*(1\otimes x)W)$ for all $x\in \C[\G]$. Assume that
  $f \in \Probc(\G)$. Recall that the KMS group of $h$ is implemented by the
  Woronowicz characters $f_z\in \C[\G]^*$ so that $h(xy) = h((f_{-1}*y*f_{-1})x)$ for all
  $x\in\Cred\G$, $y\in\C[\G]$. Denote $Q = (\id\otimes f_1)(W) = (Q_w)_w$, which is a
  positive unbounded multiplier of $c_0(\G)$. We will invoke Sweedler notation and write $W = W_{(1)}\otimes W_{(2)}$. Then we have, since
  $(S\otimes\id)(W) = W^*$ and $(S\otimes\id)(W^*) = (Q\otimes 1)W(Q^{-1}\otimes 1)$:
  \begin{align*}
    (f\otimes h)(W^*(1\otimes x)W) &= (f\otimes h)((1\otimes W_{(2)})W^*(Q^{-1}W_{(1)}Q^{-1}\otimes x)) \\
                                   &= (fS\otimes h)(W^*(Q^2\otimes 1)W(1\otimes x))
  \end{align*}
  for all $x\in\Cred\G$. Hence $f\in\Probc(\G)$ belongs to $\Probhc(\G)$ if and only if
  \begin{equation} \label{eq_prob_h}
    (fS\otimes\id)(W^*(Q^2\otimes 1)W) = 1.
  \end{equation}
  If this holds, we can roll back the computation with any $\sigma$-KMS state $\tau$
  instead of $h$, which yields the result.
\end{proof}

From~\eqref{eq_prob_h} we also see that we have $\Probhc(\G) = \Probc(\G)$ {\bf iff} $W^*(Q\otimes 1)W = 1$ {\bf iff}
$Q=1$, i.e.\ $\G$ is unimodular --- this has been known at least since \cite{I02}. Note that $\Probc(\G)$ is norm dense in $\Prob(\G)$, but the corresponding
result for $\Probh(\G)$ is not clear.

\begin{defn}\label{PAP Def}
  We say that $\G$ has the {\bf Powers averaging property} \PAP, resp.{}
  \PAPh, if for every $x = x^*\in \Cred\G$ we have
  \begin{displaymath}
    \inf_{f\in P}\|x*f - h(x)1\| = 0,
  \end{displaymath}
  where $P = \Probc(\G)$, resp.{}
  $\Probhc(\G)$. Note that
  \PAPh$\implies$\PAP.
\end{defn}

\begin{rem}
Let $\G = G$ be a discrete group. The classical \PAP, as formulated by Powers in \cite{P75}, and the formulation used in \cite{K20,H16}, is stated as follows: for every $x = x^*\in \Cred G$,
    \begin{displaymath}
      h(x)1 \in \overline\conv\{ \lambda(s)x\lambda(s)^* : s\in G \}.
    \end{displaymath}
    We can see that this formulation of the \PAP{} is equivalent to the one given
    in Definition \ref{PAP Def}. Indeed, for $s\in G$ and $x\in \Cred G$,
    $x*\delta_s = \lambda(s)x\lambda(s)^*$ and so
    \begin{displaymath}
      \overline\conv\{ \lambda(s)x\lambda(s)^* : s\in G \} = \overline{x*\Prob(G)}.
  \end{displaymath}
\end{rem}

In Definitions~\ref{Probh Def} and~\ref{PAP Def} we could have used the coaction $\ad^\op$ of $c_0(\G^\op)$, instead of $\ad$. Let us show that the resulting notions would have been the same. For this we denote $*^\op$ the convolution products associated with $\ad^\op$, we denote
 $\Probhc(\G^\op) = \{f\in\Probc(\G) : f*^\op h = h\}$ and we say e.g.\ that $\G^\op$ has the \PAPh{} if $\inf \{\|x*^\op f - h(x)1\| : f\in\Probhc(\G^\op)\} = 0$.

\begin{lem}
    A discrete quantum group $\G$ has the \PAP, resp.\ 
    \PAPh{} iff $\G^\op$ has the \PAP, resp.\ 
    \PAPh.
\end{lem}
\begin{proof}
    According to the formula $\ad^\op = (R\otimes\hat R)\circ\ad\circ\hat R$ we have 
    \begin{displaymath}
        x*^\op f := (f\otimes\id)\ad^\op(x) = \hat R(\hat R(x) * fR).
    \end{displaymath}
    Since $\hat R$ is a $*$-antiautomorphism such that $h\circ\hat R = h$, it follows that $\|x*^\op f-h(x)1\| = \|x'*fR-h(x')\|$ for $x' = \hat R(x)$. It clearly follows that the \PAP{} is the same for $\G$ and $\G^\op$. For the \PAPh{} it remains to check that $\{f\circ R : f\in\Probhc(\G^\op)\} = \Probhc(\G)$, which again follows from the connection between $\ad$ and $\ad^\op$ and the property $h\circ\hat R = h$.
\end{proof}

\begin{lem}\label{PAP Reduction}
  A DQG $\G$ has the \PAP, resp.\ 
  the \PAPh, if and only if for every
  $x = x^*\in \C[\G]\cap \ker(h)$ we have
  \begin{displaymath}
    \inf_{f\in P}\|x*f\| = 0,
  \end{displaymath}
  where $P = \Prob(\G)$, resp.\
  $\Probhc(\G)$.
\end{lem}
  
\begin{proof}
  The direct implication is obvious. For the reverse one, we first restrict from
  $\Cred\G$ to $\C[\G]$ by density and because $\|x*f\|\leq \|x\|$ for any
  $x\in\Cred\G$, $f\in\Prob(\G)$. Then we apply the assumption to $x-h(x)1$
  and observe that $(x-h(x)1)*f = x*f-h(x)1$.
\end{proof}

In the case of the \PAPh{} we can remove the limitation to self-adjoint elements.

\begin{lem}\label{PAP Reduction to Self Adjoints}
  A DQG $\G$ has the \PAPh{} if and only if for every $x \in \Cred\G$,
  \begin{displaymath}
    \inf_{f\in P}\|x*f - h(x)1\| = 0,
  \end{displaymath}
  where $P = \Probhc(\G)$.
\end{lem}

\begin{proof}
  Assume that $\G$ has the \PAPh. Let $x = x_1 + \ii x_2\in \Cred\G$
  with $x_1=x_1^*$, $x_2=x_2^*$ and fix $\epsilon > 0$. Find $f_1\in P$ such
  that $\|x_1*f_1 - h(x_1)\| < \epsilon$.  Next, put $x_2' = x_2*f_1$ and find
  $f_2\in \Probh(\G)$ such that $\|x_2'*f_2 - h(x_2')1\| < \epsilon$. Let
  $f = f_1*f_2$, which is still in $P$. Since $h(x_2') = (f_1*h)(x_2) = h(x_2)$
  and $x_2'*f_2 = x_2*f$, we have $\|x_2*f-h(x_2)1\|<\epsilon$.  On the other
  hand, since $(f_2\otimes \id)\ad$ is ucp, hence contractive, we have
  \begin{displaymath}
    \|x_1*f - h(x_1)1\| = \|(x_1*f_1 - h(x_1)1)*f_2\|
    \leq \|x_1*f_1 - h(x_1)1\| < \epsilon.
  \end{displaymath}
  As a result 
  $\|(x_1 + \ii x_2)*f - h(x_1 + \ii x_2)1\| < 2\epsilon$.  The converse is
  obvious.
\end{proof}

\bigskip

We now verify that our main example $\F U_F$ has the \PAPh.  Banica proved in \cite{B97} that $\F U_F$ is $C^*$-simple with a unique $\sigma$-KMS state,
using an adaptation of Powers' averaging techniques on $\Corep(\F U_F)$. It should be noted that Banica also defined an analogue of Powers' averaging property
for DQGs, see \cite[Section~7, Definition~3]{B97}, which turns out to not be equivalent to ours. In fact, Banica remarked in \cite[Section 8]{B97} that $U^+_F$ does not have the PAP in his sense.
However, as we are about to see, what Banica actually proves is that it does have the PAP in our sense.

At heart of the proof of \cite[Theorem 3]{B97} is indeed the use of maps $\AD(u) : \C[\G] \to \C[\G]$ from \cite[Lemma 9]{B97}, for $u\in \Corep (\G)$, which
are given by the formula
\begin{gather*}
  \AD(u)(z) = (\Tr\otimes\id)[(Q_u^b\otimes 1)u(1\otimes z)u^*(Q_u^{-d}\otimes 1)]
\end{gather*}
where $b$, $d$ are real parameters. In other words and with our notation, $\AD(u)(z) = (\qTr_u^{b-d}\otimes\id)\ad^\op(z)$ $= z \astop \qTr_u^{b-d}$, where
$\qTr_u^{b-d}(a) := \Tr_u(Q_u^{b-d}a_u)$ for $a = (a_u)_u\in c_0(\F U_F)$. Note in particular that $\qTr_u^1 = \qTr_u$ and $\qTr^{-1}_u = \qTr'_u$.

\begin{prop}[\cite{B97}]\label{Banicas Result}
    The discrete quantum groups $\F U_F$ have the \PAPh.
\end{prop}

\begin{proof}
  We claim that in the proof of \cite[Theorem 3]{B97} what Banica really showed
  is that $\F U_F^\op$ has the \PAPh{} (in our sense) and hence $\G = \F U_F$ has the \PAPh. Fix
  $z = z^* \in \C[\G]\cap \ker(h)$ and $\epsilon > 0$. In the proof of
  \cite[Theorem 3]{B97}, a unital linear map $F = V\circ W$, with $V$,
  $W : \Cred\G\to \Cred\G$, is constructed. Let us first describe the maps $V$
  and $W$.
  
  The unital map $W$ is constructed at \cite[Corollary 4]{B97} and given by
  $W = M^{-1}\AD(r)$, for a specific corepresentation $r\in I$ (a large enough power of $\bar uu$), the parameters
  $d = -b = \frac 12$ and $M>0$. The map $V$ is obtained from \cite[Corollary
  3]{B97}. Here, $V$ is of the form $V = T^m$, for some integer $m\geq 0$ and a
  unital map $T : \Cred\G\to\Cred\G$ obtained from \cite[Proposition
  8]{B97}. Similarly to $W$, $T$ is constructed from \cite[Lemma 9]{B97} and
  given by the formula
  $T = M_1^{-1} \AD(r_1) + M_2^{-1}\AD(r_2) + M_3^{-1}\AD(r_3)$ for some
  well-chosen corepresentations $r_1$, $r_2$, $r_3\in I$, the parameters
  $d = -b = 1/2$ and with $M_1$, $M_2$, $M_3>0$ (more precisely $r_i = \bar uu^i\bar u$).

  With our notation we have $T(z) = z\astop\varphi_1$ with $\varphi_1 = M_1^{-1}\qTr'_{r_1}+ M_2^{-1}\qTr'_{r_2} + M_3^{-1}\qTr'_{r_3}$ and $W(z) = z\astop \varphi_2$ with
  $\varphi_2 = M^{-1}\qTr'_r$. As noted in the Preliminaries, the maps $\qtr'_w$ belong to $\Probhc(\G^\op)$. Altogether, the unital
  map $F : \Cred\G\to \Cred\G$ is of the form $F(z) = z \astop \varphi$ for some $\varphi \in \Probhc(\G^\op)$. Moreover, $r$, $r_1$, $r_2$, $r_3$
  and $m$ are chosen so that, for our fixed $z\in\C[\G]$, one has $\|F(z)\| \leq \epsilon \|z\|$. The result thus follows from Lemma~\ref{PAP Reduction}.
\end{proof}

\bigskip

Such as in the classical case, the PAP easily implies $C^*$-simplicity and, in
the unimodular case, the unique trace property. Note that by \cite{KKSV22,AS22}
a state on $\Cred\G$ is $\G$-invariant {\bf iff} it is KMS with respect to
the scaling automorphism group $\tau$ {\bf iff} it is tracial.

\begin{prop}\label{PowersAveSimplicity}
    Let $\G$ be a DQG. If $\G$ has the \PAP{} then $\G$ is $C^*$-simple and
    \begin{enumerate}
    \item if $\G$ is unimodular then $h$ is the only $\G$-invariant state on $\Cred\G$;
    \item if $\G$ is non-unimodular then $\Cred\G$ has no $\G$-invariant states.
    \end{enumerate}
    If moreover $\G$ has the \PAPh{}, then $h$ is the unique $\sigma$-KMS
    state.
\end{prop}

\begin{proof}
  Assume $\G$ has the PAP and let $I\subseteq \Cred\G$ be a closed two-sided
  ideal. As already noted in the proof of Proposition~\ref{SimplicityCondition},
  $\ad$ stabilizes $I$, in particular $x*f\in I$ for all $x\in I$,
  $f\in\Prob(\G)$.  If $I\neq 0$ we can find $x \in I_+$ such that $h(x) = 1$,
  since $h$ is faithful on $\Cred\G$. By the PAP there exists a net of states
  $f_\alpha$ in $\Prob(\G)$ such that $x*f_\alpha \to 1$, hence $1\in I$ and $I=\Cred\G$.

  Let $\tau\in S(\Cred\G)$ be a $\G$-invariant state, i.e.\ we have
  $f*\tau = \tau$ for all $f$ in $\Prob(\G)$ or, equivalently, in
  $\Probc(\G)$. Take $x=x^* \in\Cred\G$ and a corresponding net of states
  $f_\alpha$ given by \PAP. Then we have
  $\tau(x) = (f_\alpha*\tau)(x) = \tau(x*f_\alpha) \to h(x)\tau(1) =
  h(x)$. Since the self-adjoint elements span $\Cred\G$, this shows that
  $\tau = h$. If we assume \PAPh, by Lemma~\ref{KMS States and h-Invariant
    States} this computation works also for all $\sigma$-KMS states $\tau$.
  
  If $\G$ is unimodular, tracial states on $\Cred\G$ are $\G$-invariant, hence
  $h$ is the only trace. If $\G$ is not unimodular, $h$ is not $\G$-invariant
  (i.e.\ $\Probhc(\G) \neq \Probc(\G))$ hence the previous computation shows
  that $\Cred\G$ has no $\G$-invariant states, and if $\G$ has the \PAPh, it
  shows that $h$ is the only $\sigma$-KMS state.
\end{proof}
  
\begin{rem}
  The Dixmier property of a $C^*$-algebra $A$ holds whenever for every $a\in A$,
  $\overline\conv\{ uau^* : u\in \fU(A)\} \cap Z(A)\neq \{0\}$. In the case of a classical group $C^*$-algebra $A = C^*_r(G)$, the \PAP{} clearly implies the Dixmier property, using unitaries $u = \lambda(s)$, $s\in G$. This is not so clear anymore in the quantum case. However, it was
  established in \cite{HZ84} that if $A$ is simple, then $A$ has the Dixmier
  property if and only if $A$ has at most one tracial state. Hence
  Proposition~\ref{PowersAveSimplicity} together with \cite[Theorem~3.15]{AS22} shows that if
  $\G$ has the \PAP{} then $\Cred\G$ has the Dixmier property.
\end{rem}

\subsection{The PAP and the state space of $\Cred\G$}

Using Hahn-Banach one can rephrase the \PAPh{} in terms of functionals on
$\Cred\G$. This generalizes \cite[Proposition 6.1]{K20} and part of
\cite[Theorem 4.5]{H16}.

\begin{defn}\label{Envelopes Notation}
  For $\mu\in\Cred\G^*$ we denote $K(\mu)$, resp.\
  $K_h(\mu)$, the
  space $\overline{P*\mu}^{w^*}$, where $P = \Probc(\G)$, resp.\ 
  $\Probhc(\G)$.
\end{defn}

\begin{lem}\label{PowersAveLemma}
  Let $\G$ be a DQG. The following are equivalent:
  \begin{enumerate}
  \item $\G$ has the \PAPh;
  \item for every $\mu\in S(\Cred\G)$, $h\in K_h(\mu)$;
  \item for every $\mu\in \Cred\G^*$, $\mu(1)h\in K_h(\mu)$.
  \end{enumerate}
  The equivalence between 1.\ and 2.\ also holds for the \PAP{} using $K(\mu)$ instead of $K_h(\mu)$.
\end{lem}
  
\begin{proof}
  Write $P = \Probc(\G)$ resp.\ $P = \Probhc(\G)$.
  
  $(1\implies 2)$. Take $\mu\in S(\Cred\G)$. Using the Hahn-Banach separation
  theorem, if $h\notin \overline{P*\mu}^{w^*}$ then there exists $x\in \Cred\G$
  such that
  \begin{displaymath}
    \inf_{f\in P} \Re((f*\mu)(x) - h(x)) > 0.
  \end{displaymath}
  Since $\mu = \bar\mu$, $\bar h = h$ and $\bar P = P$, the same holds for $x^*$, hence 
  we can assume that $x^* = x$ and the \PAP{} resp.\ the \PAPh{} fails because
  $\Re((f*\mu)(x) - h(x)) \leq \|x*f-h(x)1\|$.

  $(2\implies 1)$. For a contradiction, assume that the \PAP{}, resp.\ the \PAPh{}
  fails. Let $x\in \Cred\G$, $x=x^*$, be such that $\inf_{f\in P}\|x*f - h(x)1\| > 0$. By
  the Hahn-Banach separation theorem, there exists $\epsilon > 0$ and
  $\mu \in \Cred\G^*$ such that
  $\Re((f*\mu)(x)) \geq \Re(\mu(1)h(x)) + \epsilon$ for all
  $f\in P$. Again we can assume that $\mu = \bar\mu$, hence $\mu\in S(\Cred\G)$ after renormalizing, and this contradicts $2$.
  
  $(3\implies 2)$ is obvious. 
  
  $(2\implies 3)$. Let $\mu : \Cred\G \to \C$ be a bounded functional. Decompose
  $\mu = \mu_1 - \mu_2 + \ii (\mu_3 - \mu_4)$ with $\mu_i \in \Cred\G^*_+$. By
  assumption, there exists a net $(f_\alpha)\subseteq P$ such that
  $f_\alpha *\mu_1\to^{w^*} \mu_1(1)h$. Using $w^*$-compactness of
  $\overline{P*\mu_i}^{w^*}$ for each $i$, we find $\nu_i\in \Cred\G^*_+$ such
  that after passing to a convergent subnet we have
  $\lim^{w^*}_\alpha f_\alpha*\mu_i = \nu_i$ (in particular $\nu_i(1) =
  \mu_i(1)$). This shows that
  $\mu_1(1)h - \nu_2 + \ii (\nu_3 - \nu_4)\in \overline{P*\mu}^{w^*}$. Using
  $P$-invariance of $h$, and by repeating the same argument, we deduce that
  $\overline{P*\mu}^{w^*}$ contains
  $\mu_1(1)h - \nu_2(1)h + \ii (\nu_3(1)h - \nu_4(1)h) = \mu(1)h$.
\end{proof}

The preceeding Lemma motivates the following definition and leads to
Theorem~\ref{GBoundariesPAP}.

\begin{defn}\label{Boundary Envelopes Def}
  We say that $X\subseteq S(\Cred\G)$ is a $\G$, resp.\ $\G_h$-{\bf
    boundary envelope} if $X\neq\emptyset$ and $X = K(\mu)$, resp.\ 
  $K_h(\mu)$, for every $\mu\in X$.  We say that such a boundary envelope
  $X$ is {\bf non-trivial} if $X\neq \{h\}$.
\end{defn}

\begin{rem}
  Since $\Probc(\G)$ and $\Probhc(\G)$ are convex, boundary
  envelopes are automatically convex and $w^*$-closed. In the classical case
  $G$-boundary envelopes correspond to $G$-boundaries inside the space
  $S(\Cred G)$, as shown in the next proposition. In the quantum case, it does
  not make sense to consider $\G$-boundaries inside $S(\Cred\G)$, since
  $\G$-boundaries are in general noncommutative. However {\em boundary
    envelopes}, as introduced in the previous definition, will prove useful also
  in the quantum case.
\end{rem}

\begin{prop}\label{GState Boundaries Classical Case}
    Let $\G = G$ be a discrete group. The following are equivalent for $X\subseteq S(\Cred G)$:
    \begin{enumerate}
        \item $X$ is a $G$-boundary envelope;
        \item $X = \overline{\conv}^{w^*}(Y)$ where $Y\subseteq S(\Cred G)$ is a $G$-boundary;
        \item $X = \Psi^*(S(C(\partial_F G)))$ for some $G$-equivariant ucp map
          $\Psi : \Cred G\to C(\partial_FG)$.
    \end{enumerate}
  \end{prop}
  
\begin{proof}
  The proof comes from a careful inspection of the proof of \cite[Proposition~3.1]{K20}.

  $(1\implies 2)$. It follows from the definitions that $X$ is a minimal affine
  $G$-space. Let $\mathop{\mathrm{ext}}(X)$ be the set of extreme points of $X$. By
  \cite[Theorem III2.3]{Glasner}, $\overline{\mathop{\mathrm{ext}}(X)}^{w^*}$ is a
  $G$-boundary.

  $(2\implies 3)$. This is exactly the converse of \cite[Proposition 3.1]{K20}.

  $(3\implies 1)$. Because $C(\partial_F G)$ is a $G$-boundary, for every
  $\mu\in S(C(\partial_FG))$, $\overline{\Prob(G)*\mu}^{w^*} = S(C(\partial_F G))$. It is immediate from here
  that $X$ is a $G$-boundary envelope.
\end{proof}
  
\begin{lem}\label{MinimalGBoundary}
  For any $\mu\in S(\Cred\G)$, there exists a $\G$-boundary envelope $X\subseteq K(\mu)$, and 
  a $\G_h$-boundary envelope $X\subseteq K_h(\mu)$.
\end{lem}

\begin{proof}
  Let $(X_j)_{j\in J}$ be a descending net of non-empty $w^*$-closed
  $\Prob(\G)$-invariant subspaces of $K(\mu)$. By $w^*$-compactness of
  $K(\mu)$, the finite intersection property implies that $\bigcap_{j\in J}X_j$
  is a non-empty $w^*$-closed $\Prob(\G)$-invariant subspace of $K(\mu)$. By
  Zorn's lemma, minimal such subspaces exist, and they are clearly
  $\G$-boundary envelopes. The argument is the same for $K_h(\mu)$.
\end{proof}

We can then prove a quantum version of \cite[Theorem 3.6]{K20}, where
$G$-boundaries are replaced with $\G_h$-boundary envelopes. In the separable case, we also obtain a stationary dynamical characterization of the \PAPh{} which generalizes \cite[Theorem 5.1]{HK23} to the quantum setting.

\begin{thm}\label{GBoundariesPAP}
  Consider the following properties of discrete quantum group $\G$:
  \begin{enumerate}
      \item $\G$ has the \PAPh{};
      \item the only $\G_h$-boundary in $S(C^*_r(\G))$ is trivial;
      \item there exists $f\in \overline{\Prob^c_h(\G)}$ such that $h$ is the unique $f$-stationary state on $C^*_r(\G)$.
  \end{enumerate}
  We have that $(1)\iff (2) \impliedby (3)$. If $C^*_r(\G)$ is separable then $(1)\implies (3)$.
\end{thm}

\begin{proof}
  $(1)\iff (2)$ Assume that $\G$ has the \PAPh{} and let $X$ be a $\G_h$-boundary
  envelope. Take $\mu\in X$, then by Lemma \ref{PowersAveLemma} we have
  $h\in K_h(\mu) = X$. But then we have as well $X = K_h(h) =
  \{h\}$. Conversely, fix $\mu\in S(\Cred\G)$. Using Lemma
  \ref{MinimalGBoundary} we can find a $\G_h$-boundary envelope
  $X\subseteq K_h(\mu)$, and by assumption we must have $X = \{h\}$. In
  particular $h\in K_h(\mu)$ and the \PAPh{} follows from Lemma
  \ref{PowersAveLemma}.

  $(3)\implies (1)$ Let $f\in \Prob^c_h(\G)$ be such that $h$ is the unique $f$-stationary state on $C^*_r(\G)$. By the same proof as \cite[Proposition 4.7]{HK23} (with left actions swapped with right ones) we have that $||\frac{1}{n}\sum^n_{k=1}a*f^{*k} - h(a)1|| \to 0$ as $n\to\infty$ for every $a\in C^*_r(\G)$. Hence $\G$ has the \PAPh{}.

    Now assume $C^*_r(\G)$ is separable.
    
  $(1)\implies (3)$ Consider the proof of \cite[Theorem 5.1]{HK23}. By the same reasoning, we obtain a state $f\in \Prob(\G)$ such that $h$ is uniquely $f$-stationary. In fact, $f\in \overline{\Prob^c_h(\G)}$ as, by its construction, it belongs to the closed convex envelope of $\Prob^c_h(\G)$.
\end{proof}

\subsection{The PAP and ucp maps on $\Cred\G$}

As in the classical case we establish the following correspondence between $\G$-boundary envelopes and $\G$-equivariant ucp maps $\Cred\G \to
C(\partial_F\G)$. Note that by Lemma~\ref{MinimalGBoundary} every $\G$-boundary envelope contains a $\G_h$-boundary envelope, so that $\G_h$-boundary envelopes
are also connected to $\G$-equivariant ucp maps $\Cred\G \to C(\partial_F\G)$.

\begin{defn}
  For a $\G$-equivariant ucp map $\Psi : \Cred\G \to B$ we denote $X_\Psi = \Psi^*(S(B)) \subseteq S(\Cred\G)$.
\end{defn}

\begin{lem}\label{GBoundariesLem}
  Let $\G$ be a DQG. The map $\Psi \mapsto X_\Psi$ is a bijection between $\G$-equivariant ucp maps $\Cred\G\to C(\partial_F\G)$ and $\G$-boundary envelopes in $S(\Cred\G)$.
\end{lem}

\begin{proof}
  Let $B$ be a $\G$-boundary, and $\Psi : \Cred\G \to B$ a $\G$-equivariant ucp map. Then for every $\mu\in S(B)$ we have $K(\mu) = S(B)$, see
  \cite[Lemma~4.2]{KKSV22}, and since $\Psi$ is equivariant this entails $K(\nu) = X_\Psi$ for all $\nu\in X_\Psi$, i.e.\ $X_\Psi$ is a $\G$-boundary
  envelope. Note that $\Psi^*$ is $(w^*,w^*)$-continuous and $X_\Psi$ is $w^*$-closed, as it is the continuous image of the $w^*$-compact space $S(B)$.

  Conversely, let $X\subset S(\Cred\G)$ be a $\G$-boundary envelope and take $\mu\in X$. Recall that we denote
  $\fP_\mu = (\id\otimes\mu)\ad : \Cred\G \to \ell^\infty(\G)$, so that $f\circ \fP_\mu = f*\mu$ for all $f\in\Prob(\G)$.  Choose a $\G$-equivariant ucp
  projection $P : \ell^\infty(\G) \to C(\partial_F\G)$ and set $\Psi = P\circ \fP_\mu : \Cred\G\to C(\partial_F\G)$. Since
  $\varphi\circ P\in \overline{\Prob(\G)}^{w^*}$ for any $\varphi \in S(C(\partial_F\G))$ we have
  \begin{displaymath}
    X_\Psi = \{\varphi\circ \Psi : \varphi\in S(C(\partial_F\G))\}\subseteq
    \overline{\{f\circ \fP_\mu : f\in \Prob(\G)\}}^{w^*} = K(\mu) = X.
  \end{displaymath}
  Moreover, taking $\nu\in X_\Psi$, we have $X = K(\nu)$ because $X$ is a $\G$-boundary envelope, hence $X\subset X_\Psi$ because $X_\Psi$ is
  $\G$-invariant.

  On the other hand, if we started from $X = X_\Phi$ in this construction, we have $\mu = \varphi\circ\Phi$ with $\varphi\in S(C(\partial_F\G))$, hence
  $\Psi = P\circ \fP_\varphi\circ\Phi$ by equivariance of $\Phi$. But by rigidity of $\partial_F\G$, the equivariant ucp map
  $P\circ \fP_\varphi : C(\partial_F\G) \to C(\partial_F\G)$ must be the identity, hence $\Psi = \Phi$.
\end{proof}

\begin{rem}
  There is a clear analogue of a $\G$-boundary envelope in the state space of an arbitrary $\G$-$C^*$-algebra. In Lemma \ref{GBoundariesLem} and its proof, we can replace $\Cred\G$ with any $\G$-$C^*$-algebra --- to be precise, there is a bijection between $\G$-boundary envelopes in
  $S(A)$ and $\G$-equivariant ucp maps $A\to C(\partial_F\G)$ for any $\G$-$C^*$-algebra $A$.
\end{rem}

\begin{defn}\label{Factor Haar State Def}
  Given a $C^*$-algebra $A$, we say that a ucp map $\Psi : \Cred\G\to A$ factors $h$ if $h\in\Psi^*(S(A))$.
\end{defn}

\begin{lem}\label{Weak* Approximately Equal Lemma}
  Let $\G$ be a DQG and $A$ a $C^*$-algebra. A ucp map $\Psi : \Cred\G\to A$ factors $h$ {\bf iff} for all $x\in\C[\G]$ we have $h(x)\in \overline{\{\mu(\Psi(x)) : \mu\in S(A) \}}$.
\end{lem}

\begin{proof}
  Since $S(A)$ is $w^*$-compact and convex, $\Psi^*(S(A))$ is a $w^*$-closed convex subset of $S(\Cred\G)$. If $h\notin\Psi^*(S(A))$, the Hahn-Banach separation
  theorem shows that there exists $x\in \Cred\G$ and $\epsilon > 0$ such that $\Re(\mu(\Psi(x))) \geq \Re(h(x)) + 3\epsilon$ for all $\mu\in S(A)$. Taking $y\in\C[\G]$ such that $\|x-y\|\leq\epsilon$ we have $|\mu(\Psi(y)) - h(y)|\geq \epsilon$ for all $\mu\in S(A)$. The reverse direction is obvious. Note that $\{\mu(\Psi(x)) : \mu\in S(A) \}$ is in fact closed.
\end{proof}

\begin{lem}\label{Gequivariance and Weak* Approximation}
  Let $A$ be a $\G$-boundary and $\Psi : \Cred\G\to A$ a $\G$-equivariant ucp map.
  \begin{enumerate}
  \item Assume $\G$ has the \PAP. Then $\Psi$ factors $h$.
  \item Assume $\G$ is unimodular.  Then $\Psi$ factors $h$ {\bf iff} $\Psi = 1h$.  
  \end{enumerate}
\end{lem}

\begin{proof}
First we note that by (the beginning of the proof of) Lemma~\ref{GBoundariesLem} the set $X_\Psi = \Psi^*(S(A))$ is a $\G$-boundary envelope.

  1. By Lemma~\ref{PowersAveLemma}, $h\in K(\mu) = X_\Psi$, where we have chosen some $\mu\in X_\Psi$.
  
  2. If $\Psi$ factors $h$, $X_\Psi$ contains $h$, hence $X_\Psi = K(h)$. Since $\G$ is unimodular, $K(h) = \{h\}$. As a result we have $\mu\circ\Psi = h = \mu\circ (1h)$ for all $\mu\in S(A)$, hence $\Psi = 1h$. The converse is trivial.
\end{proof}

\begin{cor}\label{PAPUniqueGMap}
    A DQG $\G$ has the \PAP{} {\bf iff} every $\G$-equivariant ucp map $\Cred\G\to C(\partial_F\G)$ factors $h$. In particular, if $\G$ is unimodular then $\G$ has the \PAP{} {\bf iff} $1h$ is the unique $\G$-equivariant ucp map $\Cred\G\to C(\partial_F\G)$.
\end{cor}

\begin{proof}
    The direct implication is a particular case of Lemma \ref{Gequivariance and Weak* Approximation}. For the converse, take $\mu\in S(\Cred\G)$. By Lemma~\ref{MinimalGBoundary} there exists a $\G$-boundary $X\subset K(\mu)$, which can be written $X = X_\Psi$ for some $\G$-equivariant ucp map $\Psi : \Cred\G \to C(\partial_F\G)$ by Lemma~\ref{GBoundariesLem}. By hypothesis $\Psi$ factors $h$, hence we have $h\in X \subset K(\mu)$, and Lemma \ref{PowersAveLemma} shows that $\G$ has the \PAP. The unimodular case follows immediately from Lemma \ref{Gequivariance and Weak* Approximation} and the general case.
\end{proof}

We end this section by stating explicitly the new results that we obtain for the unitary free quantum groups.

\begin{cor}
  The only $(\F U_F)_h$-boundary envelope in $S(\Cred{\F U_F})$ is trivial. Every $\F U_F$-equivariant ucp map
  $\Cred{\F U_F}\to C(\partial_F \F U_F)$ factors $h$, and if $F$ is unitary $1h$ is the unique such map.
\end{cor}

\section{Freeness and Faithfulness of boundary actions}

\subsection{Freeness}

The following definition was formulated by Masuda and Tomatsu \cite[Definition 2.7]{MT07} in the case when $A$ is a von Neumann algebra equipped with a cocycle
action of $\G$. Recall that $W^*$-coactions on a von Neumann algebra $A$ are also $C^*$-coactions on $A$, see the discussion in \cite[Section
2.5]{KKSV22}. The same notion was also considered in \cite{Z19} in the case when $\G = G$ is a classical discrete group. See also \cite{Kallman_Free,CKN_Dependent} in the case of a single automorphism. One should be aware that in the case of classical actions of groups on locally compact spaces one recovers the notion of {\em topological} freeness ; we shall nevertheless keep the terminology of \cite{MT07} and \cite{Z19}.

In this section we denote $p_0\in c_c(\G)$ the support projection of the counit, which coincides with the minimal central projection corresponding to the trivial corepresentation $1\in I$. Recall that for a coaction $\alpha : A \to M(c_0(\G)\otimes A)$, $a\in A$ and $u\in I$ we denote $\alpha_u(a) = (p_u\otimes 1)\alpha(a) \in B(H_u)\otimes A$.

\begin{defn}\label{Topological Freeness}
  We say that $\G\acts A$ is {\bf free} if for any $X\in M(c_0(\G)\otimes A)$, $X(1\otimes a) = \alpha(a)X$ for all $a\in A$ implies
  $X \in p_0\otimes Z(A)$. Equivalently, for every $1\neq u\in I$ and every $X\in B(H_u)\otimes A$,
    $[~\forall a\in A \quad X(1\otimes a) = \alpha_u(a)X~] \implies X = 0$.
\end{defn}

\begin{prop}
  Let $\G = G$ be a discrete group acting on a $G$-$C^*$-algebra $A$. We have that $G\acts A$ is free (in our sense) if and only if $G$ acts freely on $A$ (in the
  sense of \cite{Z19}), i.e.\ $\alpha_g(a)b = ba$ for all $a\in A$ and some $g\in G\setminus \{e\}$ implies $b = 0$. In particular, if $A = C(X)$ is commutative
  then $G \acts C(X)$ is free if and only if $G\acts X$ is topologically free.
\end{prop}

\begin{proof}
  Since $B(H_g) \simeq \C$ for every $g\in I = G$, the equation $X(1\otimes a) = \alpha(a)X$, for $X\in B(H_g)\otimes A$, is equivalent to $Xa = \alpha_g(a)X$ for every
  $g\in G$. This establishes the claim that $G\acts A$ is free if and only if $G$ acts freely on $A$.

  Now, assume $G\acts C(X)$ is free. Take $g\in G$ and an open subset $U\subseteq \Fix(g)$. Take any $b\in C(X)$ whose support is
  contained in $U$. Then for any $a\in C(X)$ we have $\alpha_g(a)b = ab$, hence $b=0$. This shows that $\Fix(g)$ has empty interior as desired. Conversely, suppose that $G\acts C(X)$ is not free. Then there exists $b\in C(X)$, $b\neq 0$ and $g\in G$, $g\neq e$ such that $b\alpha_g(a) = ab$ for every $a\in C(X)$. Let $U$ be a non empty open subset contained in the
  support of $b$. Then $U\subseteq \Fix(g)$.
\end{proof}

Compare the following with \cite[Lemma 2.8]{MT07}, \cite[Theorem 2.14]{MT07}, and \cite[Section 7.4]{KS19}. Our theorem is a quantum analogue of \cite[Theorem
3.2]{Z19} and the statement $2.\iff 3.$ of our theorem is a $C^*$-algebraic analogue of \cite[Theorem 2.14]{MT07}.

Recall that $A\rtimes_r\G$ is the closed subspace of $M(K(\ell^2(\G))\otimes A)$ generated by elements $(x\otimes 1)\alpha(a)$ where $x\in\Cred\G$ and $a\in A$. We
can and will identify $\Cred\G$ and $A$ with subalgebras of $M(A\rtimes_r\G)$ --- moreover in our setting $A\rtimes_r\G$ will be unital. We denote
$E_0 = \omega_{\xi_0}\otimes \id : A\rtimes_r\G \to A$ the canonical conditional expectation, where $\xi_0\in \ell^2(\G)$ is the canonical $\Cred\G$-cyclic vector.

\begin{thm}\label{UniqueConditionalExpectations}
  Let $\G$ be a DQG and $A$ be a $\G$-$C^*$-algebra. The following are equivalent:
  \begin{enumerate}
  \item there exists a unique conditional expectation $A\rtimes_r\G\to A$;
  \item $\alpha(A)'\cap (A\rtimes_r\G) = Z(\alpha(A))$;
  \item $\G\acts^\alpha A$ is free.
  \end{enumerate}
\end{thm}

\begin{proof}
    $1\Rightarrow 2$. This follows from \cite[Proposition 3.1]{Z19} while using the fact that the canonical conditional expectation $E_0 : A\rtimes_r\G\to A$ is faithful.
    
    $2\Rightarrow 3$. Take $1\neq u\in I$ and $X\in B(H_u)\otimes A$ such that $X(1\otimes a) = \alpha_u(a)X$ for all $a\in A$. Denote
    $X^\alpha = (\id\otimes\alpha)X\in B(H_u)\otimes M(c_0(\G)\otimes A)$. Then
    \begin{align*}
      X^\alpha (1\otimes \alpha(a)) &= ((\id\otimes\alpha)\alpha_u(a)) X^\alpha = (p_u\otimes 1\otimes 1) W^*_{12} \alpha(a)_{23} W_{12} X^\alpha \\
      &= (u^*\otimes 1)(1\otimes \alpha(a)) (u\otimes 1) X^\alpha,
    \end{align*}
    where we identify $p_uc_0(\G)$ with $B(H_u)$. By assumption we obtain
    \begin{displaymath}
      (u\otimes 1)X^\alpha \in (1\otimes \alpha(A))'\cap (B(H_u)\otimes (A\rtimes_r\G)) = B(H_u)\otimes Z(\alpha(A)),
    \end{displaymath}
    hence $X^\alpha \in (u^*\otimes 1)(B(H_u)\otimes \alpha(A))$. Applying the canonical conditional expectation $E_0$ on the second leg of
    $B(H_u)\otimes (A\rtimes_r\G)$ we obtain $X = 0$, since for $u \neq 1$ we have $(\id\otimes E_0)(u^*) = 0$.
    
    $3\Rightarrow 1$. Let $E : A\rtimes\G\to A$ be a conditional expectation; we have in particular $E((x\otimes 1)\alpha(a)) = E(x)a$ and
    $E(\alpha(a)(x\otimes 1)) = aE(x)$ for $x\in \Cred\G$, $a\in A$. Denote $X = (\id\otimes E)(W^*_{12})\in M(c_0(\G)\otimes A)$, where
    $W^*_{12} \in M(c_0(\G)\otimes A\rtimes_r\G)$. For $a\in A$ we have
    $W_{12}^*(1\otimes\alpha(a)) = (\Delta\otimes\id)\alpha(a)W_{12}^* = (\id\otimes\alpha)\alpha(a)W_{12}^*$. Applying $\id\otimes E$ yields
    $X(1\otimes a) = \alpha(a)X$. By assumption this implies $X = p_0\otimes a_0$ with $a_0\in Z(A)$ ; in particular we have $(\id\otimes E)(w) = 0$ for any
    $1\neq w\in I$, thus $E = E_0$.
\end{proof}

In the special case of $A = C(\partial_F\G)$, freeness implies the PAP, and in particular, $C^*$-simplicity. However, it also implies unimodularity and so freeness is probably not the appropriate notion to consider for boundary actions of non-unimodular discrete quantum groups. Moreover it is not clear that freeness passes from $\G$-boundaries to the Furstenberg boundary as does topological freeness in the classical case.

\begin{cor}\label{TopFreenessImpliesPAPKac}
  Let $\G\acts A$ be a free action.
  \begin{enumerate}
  \item If $A$ is $\G$-injective then $\G$ is unimodular.
  \item If $A = C(\partial_F\G)$ is the Furstenberg boundary then $\G$ has the \PAP.
  \end{enumerate}
\end{cor}

\begin{proof}
  1. Assume $\G$ is not unimodular. In particular $h$ is not $\G$-equivariant on $\Cred\G$, see Section~\ref{sec_def_PAP}, hence the canonical expectation $E_0 : A\rtimes_r\G\to A$ is not
  $\G$-equivariant either. On the other hand, since $A$ is $\G$-injective, the identity map $A\to A$ extends to a $\G$-equivariant ucp projection $A\rtimes_r\G\to A$.
  Then Theorem \ref{UniqueConditionalExpectations} implies that $\G\acts A$ cannot be  free.

  2. Let $\Psi : \Cred\G\to C(\partial_F\G)$ be a $\G$-equivariant ucp map. Using the $\G$-equivariant inclusion $\Cred\G\subseteq C(\partial_F\G)\rtimes_r\G$
  and $\G$-injectivity, we obtain a $\G$-equivariant ucp extension $\tilde\Psi : C(\partial_F\G)\rtimes_r\G\to C(\partial_F\G)$. By $\G$-rigidity $\tilde\Psi $
  restricts to the identity map on $C(\partial_F\G)$, hence it is a conditional expectation. Theorem \ref{UniqueConditionalExpectations} implies
  $\tilde\Psi = E_0$, hence $\Psi = 1h$ and Theorem \ref{PAPUniqueGMap} gives us the \PAP, since $\G$ is unimodular by the first point.
\end{proof}

If $\G = G$ is a discrete group, it is a well-known and easy fact that the action $G\acts^\Delta \ell^\infty(G)$ is free. We have the following converse:

\begin{cor}\label{Topological Freeness of linfty}
  The action $\G\acts \ell^\infty(\G)$ is free if and only if $\ell^\infty(\G)$ is commutative.
\end{cor}
  
\begin{proof}
  We apply the von Neumann version of Theorem~\ref{UniqueConditionalExpectations}, namely
  \cite[Theorem~2.14]{MT07}, which says that $\G\acts \ell^\infty(\G)$ is free if and only if
  $\ell^\infty(\G)'\cap (\ell^\infty(\G)\rtimes \G)'' = Z(\ell^\infty(\G))$. Since $(\ell^\infty(\G)\rtimes \G)'' = B(\ell^2(\G))$ this is equivalent to commutativity of $\ell^\infty(\G)'$, hence of $\ell^\infty(\G)$ which is standardly
  represented on $\ell^2(\G)$.
\end{proof}

\subsection{Faithfulness}

In this subsection we will prove an analogue of Theorem~\ref{UniqueConditionalExpectations} for faithfulness of actions. The notion of faithfulness was
introduced in \cite{KKSV22} in connection with boundary actions. We give below four new characterizations of faithfulness, which allow in particular to see that
freeness implies faithfulness, by comparing e.g.\ point 3.\ of Proposition~\ref{Faithful Coactions} and the definition of freeness, or
Theorem~\ref{Faithful Coaction Theorem} and Theorem~\ref{UniqueConditionalExpectations}.

\begin{defn}[\cite{KKSV22}]
  The action $\G\acts^\alpha A$ is {\bf faithful} if $N_\alpha := \{\fP_\mu(a) : \mu\in A^*, a\in A\}'' = \ell^\infty(\G)$.
\end{defn}

Recall that the {\em cokernel} $N_\alpha$ of $\alpha$ is a Baaj-Vaes subalgebra of $\ell^\infty(\G)$ \cite[Proposition~2.9]{KKSV22}, in particular it can be
realized as the group von Neumann algebra $\ell^\infty(\H)$ of a closed quantum subgroup $\hat\H\subset\hat\G$ of the dual of $\G$
\cite[Proposition~10.5]{BV05}, and there is a {\em group-like projection} $P_\alpha\in\ell^\infty(\G)$ such that
$N_\alpha = \{f\in\ell^\infty(\G) \mid (1\otimes P_\alpha)\Delta(f) = f\otimes P_\alpha\}$ \cite[Theorem 3.1]{FK18}.

\begin{prop}\label{Faithful Coactions}
  Let $\G$ be a DQG and $A$ be a $\G$-$C^*$-algebra. TFAE:
  \begin{enumerate}
  \item $\G\acts^\alpha A$ is faithful;
  \item if $\varphi\in \ell^1(\G)$ satisfies $\varphi*\mu = \mu$ for every $\mu\in A^*$, then $\varphi=\epsilon$;
  \item if $f\in \ell^\infty(\G)$ satisfies $(f\otimes 1)\alpha(a) = f\otimes a$ for every $a\in A$, then $f \in \C p_0$;
  \end{enumerate}
\end{prop}

\begin{proof}
  $1 \Rightarrow 2$. Assume there exists $\epsilon \neq \varphi\in \ell^1(\G)$ such that $\varphi*\mu = \mu$ for every $\mu\in A^*$. Then, $\psi = \varphi-\epsilon$ is a non-zero
  element of $\ell^1(\G)$ such that $\psi*\mu = 0$ for all $\mu\in A^*$. In particular, $\psi(\fP_\mu(a)) = 0$ for all $\mu\in A^*$ and $a\in A$. This shows that
  $\psi|_{N_\alpha} = 0$ and hence $N_\alpha\neq \ell^\infty(\G)$.
    
  $2 \Rightarrow 3$. Assume that there exists $f\in \ell^\infty(\G)$, $f\notin \C p_0$, such that $(f\otimes 1)\alpha(a) = f\otimes a$ for every $a\in A$. We
  can choose $\varphi\in \ell^1(\G)$ such that $\varphi(f(1-p_0)) \neq 0$ and $\varphi(f)=1$. Then, for $\mu\in A^*$ and $a\in A$, an application of $\varphi\otimes \mu$ to the above equation implies $(\varphi f)*\mu(a) = \mu(a)$.  In other words, $(\varphi f)*\mu = \mu$ for every $\mu\in A^*$, but we have $\varphi f \neq \epsilon$.

  $3 \Rightarrow 1$. Let $P_\alpha \in N_\alpha$ be the group-like projection corresponding to $N_\alpha$. For any $f\in N_\alpha$ we have
  $(1\otimes P_\alpha)\Delta(f) = f\otimes P_\alpha$, hence $P_\alpha f = \epsilon(f)P_\alpha$. Applying this to $f = \fP_\mu(a)$, for $a\in A$, $\mu\in A^*$,
  we get $(P_\alpha\otimes\mu)\alpha(a) = \mu(a)P_\alpha$. It follows that $(P_\alpha\otimes 1)\alpha(a) = P_\alpha \otimes a$ for all $a\in A$, which by
  assumption implies $P_\alpha\in\C p_0$. Since $p_0$ is the support of the co-unit, the condition $(1\otimes P_\alpha)\Delta(f) = f\otimes P_\alpha$ becomes
  void and we have $N_\alpha = \ell^\infty(\G)$.
\end{proof}

Note that every $\G$-invariant state $\tau : \Cred\G\to \C$ extends by $\G$-injectivity to a $\G$-equivariant conditional expectation
$E : C(\partial_F\G)\rtimes_r\G\to C(\partial_F\G)$ such that $E(\Cred\G) = \C 1$. This fact was exploited in \cite{KKSV22} to show that faithfulness of
$\G\acts C(\partial_F\G)$ implies that $h$ is the only possible $\G$-invariant state on $\Cred\G$, see also \cite{AS22}. In the same spirit, our next result
establishes a characterization of faithfulness of $\G\acts A$ in terms of the uniqueness of conditional expectations $A\rtimes_r\G\to A$ that restrict to states on $\Cred\G$.

\begin{thm}\label{Faithful Coaction Theorem}
    Let $\G$ be a DQG and $A$ be a $\G$-$C^*$-algebra. TFAE:
    \begin{enumerate}
        \item there exists a unique conditional expectation $E : A\rtimes_r \G\to A$ such that $E(\Cred\G) = \C 1$;
        \item $\alpha(A)'\cap (\Cred\G \otimes 1) = \C 1$;
        \item $\G\acts^\alpha A$ is faithful.
    \end{enumerate}
\end{thm}

\begin{proof}
  The proofs of $(2 \Rightarrow 3)$ and $(3 \Rightarrow 1)$ follow from essentially the same reasoning as in the proof of Theorem
  \ref{UniqueConditionalExpectations}. The proof of $(1 \Rightarrow 2)$ follows from the same argument used in the proof of \cite[Proposition 3.1]{Z19} with
  only a minor adjustment.
    
  $1 \Rightarrow 2$. The canonical expectation $E_0 : A\rtimes_r \G\to A$ is the unique conditional expectation satisfying $E_0(\Cred\G) = h(\Cred\G) =
  \C$. Let $x\in \alpha(A)'\cap (\Cred\G\otimes 1)$ be a self-adjoint element with
  $\|x\| < 1$, so that $1 - x \in \alpha(A)'\cap (\Cred\G\otimes 1)$ is positive and invertible. Then $1 - E_0(x)\in \C 1$ is positive and invertible. Now,
  define a ucp map $\theta : A\rtimes_r\G \to A$ by setting
  \begin{displaymath}
    \theta(z) = E_0((1-x)^{1/2}z(1-x)^{1/2})(1 - E_0(x))^{-1}.
  \end{displaymath}
  It is clear that $\theta(a) = a$ for $a\in A$ and $\theta(\Cred\G) = \C$, where the latter follows because $(1-x)^{1/2}\in \Cred\G$. Therefore $\theta = E_0$ by assumption. So,
  \begin{displaymath}
    E_0(x)(1 - E_0(x)) = E_0((1-x)^{1/2}x(1-x)^{1/2}) = E_0(x - x^2)
  \end{displaymath}
  which implies that $E_0(x^2) = E_0(x)^2$ and $x$ is in the multiplicative domain of $E_0$. We have thus shown that $E_0$ restricts to a $*$-character on $\alpha(A)'\cap (\Cred\G\otimes 1)$. Since $E_0$ is faithful, this restriction is injective, hence we must have
  $\alpha(A)'\cap (\Cred\G\otimes 1) = \C 1$.

  $2 \Rightarrow 3$. Take $1\neq u\in I$ and $x\in B(H_u)$ such that $\alpha(a)(x\otimes 1) = x\otimes a$ for all $a\in A$. An application of
  $\id\otimes \alpha$ shows that
  \begin{align*}
    (x\otimes 1 \otimes 1) \alpha(a)_{23} = (\id\otimes\alpha)[\alpha(a)(x\otimes 1\otimes 1)] &= W^*_{12}\alpha(a)_{23}W_{12}(x\otimes 1\otimes 1) \\
                                                                                             &= u^*_{12}\alpha(a)_{23}u_{12}(x\otimes 1\otimes 1).
  \end{align*}
  Therefore, $(u\otimes 1)(x\otimes 1\otimes 1) \in (1\otimes \alpha(A))'$. Then for any $\varphi\in B(H_u)^*$ the element
  $((x\varphi)\otimes\id)(u)\otimes 1 \in \alpha(A)'\cap (\Cred\G \otimes 1)$ is scalar by assumption, and since $u\neq 1$ this is only possible if it vanishes, so that
  necessarily $x=0$. This shows that the third characterization of Proposition~\ref{Faithful Coactions}, is satisfied, since we must have $x := p_u f = 0$ for
  all $u\neq 1$.

  $3 \Rightarrow 1$. As in the proof of Theorem \ref{UniqueConditionalExpectations}, if $E : A\rtimes_r\G\to A$ is a conditional expectation and
  $X = (\id\otimes E)(W^*_{12})\in M(c_0(\G)\otimes A)$, the coaction equation implies $X(1\otimes a) = \alpha(a)X$ for all $a\in A$. If moreover
  $E(\Cred\G) = \C 1$ we have $X = f\otimes 1$ with $f\in \ell^\infty(\G)$ and Proposition~\ref{Faithful Coactions} implies $f\in \C p_0$. This shows that
  $E = E_0$.
\end{proof}
  
\begin{eg}
Clearly, the action of $\G$ on $\ell^\infty(\G)$ is always faithful.

For examples with $\G$-boundaries, recall Theorem~6.9 from \cite{ASK23} which proves that for $\G$ exact, $C^*$-simplicity implies the faithfulness of the action on $C(\partial_F\G)$. This shows that the action on $C(\partial_F\G)$ is faithful in the following cases: $\G = \F U_F$, $\G = \F O_F$ with $F\bar F = \pm I_N$ and $\|F\|^8\leq\frac 38 \Tr(FF^*)$, and the dual $\G$ of $\mathrm{Aut}(B,\psi)$ where $\psi$ is a $\delta$-trace and $\dim(B)\geq 8$, by \cite{B97,VV07,B13} respectively.

Note on the other hand that faithfulness of the action of $\F O_F$ on its quantum Gromov boundary $C(\partial_G \F O_F)$ is proved in \cite{KKSV22} without restriction on $F$ (subject to $F\bar F = \pm I_N$). The faithfulness of the action on $C(\partial_F \F O_F)$ follows because we have a $\G$-equivariant uci map $C(\partial_G\F O_F) \to C(\partial_F \F O_F)$.
\end{eg}

\subsection{Strong $C^*$-Faithfulness}

\begin{defn}\label{Def C*faithful}
  Let $A$ be a $\G$-$C^*$-algebra. We say $\G\acts A$ is {\bf $C^*$-faithful} if for every $\eta > 0$ and every minimal central projection $p\in c_c(\G)$ with
  $\epsilon(p)=0$ there exists $k\in\N^*$ and $b\in (A\otimes M_k(\C))_+$ such that $\|b\| = 1$ and $\|(p\otimes b)(\alpha\otimes\id)(b)\|\leq \eta$. We say that
  $\G\acts A$ is {\bf strongly $C^*$-faithful} if the same holds for every $\eta>0$ and for every {\em finite rank} central projection $p\in c_c(\G)$.
\end{defn}

\begin{rem}
  One could consider a strengthened version of (strong) $C^*$-faithfulness by allowing $\eta = 0$ --- and this is indeed what we will achieve in the case of
  the Gromov boundary of $\F U_F$ below. But we think that it would be too strong in general ; and moreover we need the restriction to $\eta>0$ to show that
  (strong) $C^*$-faithfulness passes to the universal Furstenberg boundary. We could also strengthen the definition by imposing $k=1$, but this time our
  verification for the Gromov boundary of $\F U_F$ would not work.
\end{rem}

We prove below that the notions of (strong) $C^*$-faithfulness reduce to the classical notions for actions of classical groups on classical spaces, and we make
the connection with the classical notion of topological freeness. Recall that an action $G\acts X$ is called strongly faithful if for any finite subset
$F \subset G\setminus\{e\}$ there exists $x\in X$ such that $gx \neq x$ for all $g\in F$, see e.g. \cite[Lemma~4]{delaHarpe} and \cite[Section~2.1]{FLMMS}.

\begin{prop} \label{Prop C*Faithfulness Classical}
  Assume $\G = G$ is a classical discrete group acting on a commutative $C^*$-algebra $A = C_0(X)$. Then $\G \acts A$ is (resp.\ strongly) $C^*$-faithful {\bf
    iff} $G\acts X$ is (resp.\ strongly) faithful. If the action $G\acts X$ is topologically free then it is strongly faithful, and the converse is true if $X$
  is compact and the action of $G$ is minimal.
\end{prop}

\begin{proof}
  Assume that $\G\acts A$ is strongly $C^*$-faithful and take $F\subset G\setminus\{e\}$ finite. Considering the characteristic function $p\in c_c(\G)$ of $F$
  and $\eta = \frac 14$ we obtain $b \in (A\otimes M_k(\C))_+ \simeq C_0(X,M_k(\C)_+)$ such that $\|b\| = 1$ and $\|b \times gb\| \leq \frac 14$ for all
  $g\in F$. Pick a point $x\in X$ such that $\|b(x)\| > \frac 12$. We claim that $x$ cannot be fixed by any $g\in F$: otherwise we would have
  $\|(b\times gb)(x)\| = \|b(x)b(g^{-1}x)\| = \|b(x)\|^2 > \frac 14$.

  For the reverse implication, take also $F\subset G\setminus\{e\}$ finite and $x\in X$ such that $x\neq gx$ for all $g\in F$. For each $g\in F$ choose open
  subsets $U_g$, $V_g\subset X$ such that $x\in U_g$, $g^{-1}x \subset V_g$ and $U_g\cap V_g = \emptyset$. Consider $U = \bigcap_{g\in F} U_g\cap g V_g$, which
  is still an open subset containing $x$. By construction we have $U\cap g^{-1}U = \emptyset$ for all $g\in F$. Now it suffices to take $a\in C_0(X)_+$ that
  vanishes on $X\setminus U$ and such that $\|a\| = 1$: then for any $y\in X$ and $g\in F$ at least one of $a(y)$ and $(ga)(y)$ vanishes, so that
  $a \times ga = 0$. This proves strong $C^*$-faithfulness with $\eta = 0$ and $k = 1$.

  Assume now that $G\acts X$ is topologically free and take $F\subset G\setminus\{e\}$ finite. By assumption $\Fix(g)$ has empty interior for any $g\in
  G$. Since $F$ is finite, $Y = \bigcup_{g\in F} \Fix(g)$ still has empty interior and we can find $x\in X\setminus Y$, thus proving strong faithfulness.

  Assume finally that $X$ is compact and that $G \acts X$ is minimal but not topologically free. Then there exists $g\in G$, $g\neq e$ such that $\Fix(g)$
  contains a non-empty open subset $U$. By minimality, $(h U)_{h\in G}$ is an open cover of $X$ and by compactness we can find $h_1, \ldots, h_n \in G$ such that
  $\bigcup_i h_i U = X$. We have then $\bigcup_i \Fix(h_igh_i^{-1}) = X$, and this shows that strong faithfulness fails for the finite subset
  $\{h_igh_i^{-1} ; i = 1, \ldots, n\} \subset G\setminus\{e\}$ .
\end{proof}

\begin{prop} \label{Prop C*faithful implies faithful}
  Let $A$ be a $\G$-$C^*$-algebra. If $\G\acts^\alpha A$ is $C^*$-faithful, then it is faithful.
\end{prop}

\begin{proof}
  It suffices to find, for every $\eta>0$ and every minimal central projection $p\in c_c(\G)$ such that $\epsilon(p) = 0$, an element $f\in N_\alpha$ such that
  $\|f\|\leq 1$, $p_0f = p_0$ and $\|pf\|\leq\eta$. Indeed this implies by approximation that $p_0\in N_\alpha$, which entails $N_\alpha = \ell^\infty(\G)$, see \cite{KKSV22}.
  Take the element $b\in (A\otimes M_k(\C))_+$ given by $C^*$-faithfulness (with respect to $\eta$ and $p$). By Hahn-Banach, find a state
  $\nu \in (A\otimes M_k(\C))^*$ which restrict to evaluation at $1$ on $C^*(b) \simeq C(\Sp(b))$, and put
  $f = (\id\otimes\nu)[(1\otimes b)(\alpha\otimes\id)(b)]$. We have $\|b\| = 1$, hence $\|(\alpha\otimes\id)(b)\|=1$, and since $\nu$ is a state, $\|f\|\leq
  1$. Since $p_0$ is the support of the co-unit we have moreover $p_0f = \nu(b^2)p_0 = 1^2 p_0 = p_0$, and by choice of $b$ we have
  $\|pf\|\leq \|(p\otimes b)(\alpha\otimes\id)(b)\|\leq\eta$. Finally, decomposing $\nu = \sum \mu_i\otimes \tau_i$ and $b = \sum a_j\otimes m_j$ into finite sums with
  $\mu_i\in A^*$, $\tau_i\in M_k(\C)^*$, $a_j\in A$, $m_j\in M_k(\C)$ we have $f = \sum \tau_i(m_jm_k) (\id\otimes\mu_ib_j)(\alpha(b_k)) \in N_\alpha$.
\end{proof}

\begin{rem}
Proposition~\ref{Prop C*Faithfulness Classical} shows that for classical groups acting on classical spaces, the implication proved in the previous proposition is an equivalence. This is not the case in general, already for classical groups acting on noncommutative $C^*$-algebras. Indeed, fix $\theta\in\R$ and consider the automorphism $\alpha$ of $M_2(\C)$ given by $\alpha(a)=u_\theta^*a u_\theta$, where $u_\theta = \mathop{\mathrm{diag}}(1,e^{i\theta})$. Then one can check that for any $b\in M_2(\C)\otimes B(\ell^2\N)$ positive such that $\|b\| = 1$, one has $\|b(\alpha\otimes\id)(b)\|\geq 1-|1-e^{i\theta}|$. As a result, the faithful action of $\Z$ or $\Z/N\Z$ generated by $\alpha$ is not $C^*$-faithful if $0 < |1-e^{i\theta}| < 1$. In fact, rather than $C^*$-faithfulness, the interesting notion for us is strong $C^*$-faithfulness ``in presence of minimality'' which we use as a replacement of freeness for quantum boundary actions.
\end{rem}


As a first example let us remark that $\G\acts^\Delta\ell^\infty(\G)$ is
strongly $C^*$-faithful. Indeed, the support of the counit
$p_0\in \ell^\infty(\G)$ is a norm one positive element and
$(p\otimes p_0)\Delta(p_0) = 0$ for every central projection $p\in c_c(\G)$ such
that $\epsilon(p) = 0$.

We prove now strong $C^*$-faithfulness for our central example, the quantum
Gromov boundary of $\F U_F$. The proof relies on a nice combinatorial trick
uncovered by Banica in his proof of Power's Property for $\F U_F$, see
\cite[Lemme~13]{B97}. It will then follow from Corollary~\ref{Cor FUF boundary}
and Corollary~\ref{Strongly C* faithful Implies C*simple and Faithful} that
$\F U_F$ is $C^*$-simple. As mentioned earlier, it was already established by
Banica that $\F U_F$ possesses the PAP and hence is $C^*$-simple.

Recall the notation
$C(\partial_G\F U_F) = B_\infty = B / c_0(\F U_F)$ from Section~\ref{Section
  Boundary FUF}, the associated coaction
$\beta : B_\infty\to M(c_0(\F U_F)\otimes B_\infty)$ and the central elements
$\pi_x \in B_\infty$, for $x\in I$, given by the formula
$\pi_x = \sum_{y\in I}p_{xy}$ in $B$.

\begin{prop}\label{Coaction on Gromov Boundary is Strongly Faithful}
  The action of $\F U_F$ on $C(\partial_G\F U_F)$ is strongly $C^*$-faithful.
\end{prop}

\begin{proof}
  Let $p\in c_c(\F U_F)$ be a finite-rank central projection such that $\epsilon(p) = 0$.  We have $p = \sum_{x\in F} p_x$ for a finite subset $F\subset I$ not
  containing $1$. According to \cite[Lemme~13]{B97} we can find $n\in\N$ such that all irreducible subobjects of $U\otimes x\otimes U$, with $x\in F$ and
  $U = (u\bar u)^n$, are indexed by words starting with $u$ and ending with $\bar u$, or by the empty word $1$. We consider the projection
  $b_0 = (p_U\otimes 1)\beta(\pi_{\bar u}) \in B(H_U)\otimes B \subset B(H_U)\otimes\ell^\infty(\G)$, we put $b = (Q_U^{-1}\otimes 1)b_0(Q_U^{-1}\otimes 1)$ and
  we will show that Definition~\ref{Def C*faithful} holds for $\Sigma(b)/\|b\|$, with $k = \dim U$ and $\eta = 0$. Here
  $\Sigma : B(H_U)\otimes B \to B\otimes B(H_U)$ is the flip map.

  Let us give a more explicit expression of $b_0$. Recall that
  $\beta = \Delta_{|B}$ and that $(p_x\otimes p_y)\Delta(p_z)$ identifies in
  $B(H_x\otimes H_y)$ with the projection onto the $z$-homogeneous subspace,
  which in the case of $\F U_F$ has multiplicity $0$ or $1$. Now, let $s\in
  I$ be such that $(1\otimes p_s)b_0 \neq 0$. Equivalently, $U\otimes
  s$ contains a corepresentation $z$ starting with $\bar
  u$. According to the fusion rules~(\ref{eq_fusion_rules}) this is only
  possible if $s = Uz = U\otimes z$ --- observe indeed that $U$ starts with
  $u$ and ends with $\bar u$. Then $(1\otimes
  p_s)b_0$ corresponds to the projection $P(z,U\otimes Uz) = P(1,U\otimes
  U)\otimes p_z$ onto $t_U(1)\otimes H_z$. As a result $b_0 = b_0 (1\otimes
  \pi_{U\bar u}) = P(1,U\otimes U)\otimes \pi_{\bar
    u}$, where we identify $B(H_U)\otimes \pi_{\bar
    u}\ell^\infty(\G)$ with $\pi_{U\bar
    u}\ell^\infty(\G)$ via the isomorphisms $H_U\otimes H_y \simeq
  H_{Uy}$, for $y\in I$ starting with $\bar u$.

  Now we compute $(1\otimes p_x\otimes p_s)(\id\otimes\beta)(b_0)$, for $x\in F$, $s = Uy$,
  $y$ starting with $\bar u$. This element is non-zero if and only if
  $U\otimes x\otimes s = U\otimes x\otimes U\otimes y$ contains an irreducible
  corepresentation $z$ starting with $\bar u$. By irreducibility, there exists
  then a subobject $t\subset U\otimes x\otimes U$ such that
  $z\subset t\otimes y$. If the word $t$ is non-empty, by construction of $U$ it
  starts with $u$ and ends with $\bar u$. Since $y$ starts with $\bar u$ we have
  then $t\otimes y = ty$ and $z = ty$, which is impossible because $z$ starts
  with $\bar u$. Hence $t = 1$ and $z = y$ (and $x$ has to be itself a power of
  $u\bar u$). As a result
  $(1\otimes p_x\otimes \pi_{U\bar u})(\id\otimes\beta)(b_0) = P(1,U\otimes x\otimes U)\otimes \pi_{\bar
    u} \in B(H_U)\otimes B(H_x)\otimes B(H_U)\otimes \pi_{\bar
    u}\ell^\infty(\G)$, where we identify again the last two tensor factors with
  $\pi_{U\bar u}\ell^\infty(\G)$.

  Let us now compute the product
  \begin{displaymath}
    (1\otimes p_x\otimes 1) b_{13}(\id\otimes\beta)(b) =
    (Q_U^{-1}\otimes 1\otimes 1)b_{0,13}
    (Q_U^{-2}\otimes p_x\otimes \pi_{U\bar u})(\id\otimes\beta)(b_0)(Q_U^{-1}\otimes 1\otimes 1),
  \end{displaymath}
  for $x\in F$. Using the same identification as above we have
  \begin{displaymath}
    b_{0,13} (Q_U^{-2}\otimes p_x\otimes \pi_{U\bar u})(\id\otimes\beta)(b_0) =
    [P(1,U\otimes U)_{13}(Q_U^{-2}\otimes 1\otimes 1)P(1,U\otimes x\otimes U)]\otimes \pi_{\bar u}.
  \end{displaymath}
  Note that $P(1,U\otimes U)$ is a positive multiple of $t_U t_U^*$. On the
  other hand, by Frobenius reciprocity morphisms from $1$ to $U\otimes x\otimes U$ can be written
  $(V\otimes\id)t_U$, where $V : U\to U\otimes x$ is an intertwiner. According
  to the fusion rules, such an intertwiner is unique up to a scalar, so that
  $P(1,U\otimes x\otimes U)$ is a positive multiple of
  $(V\otimes\id)t_U t_U^*(V^*\otimes\id)$. Now we observe that for $\zeta\in H_x$ we have
  \begin{align*}
    \zeta^*t_{U,13}^*(Q_U^{-2}\otimes 1\otimes 1)(V\otimes\id)t_U
    &= t_U^*(Q_U^{-2}\otimes 1)((1\otimes\zeta^*)V\otimes\id)t_U \\
    &= \qTr_U(Q_U^{-2}(1\otimes\zeta^*)V) = \qTr'_U((1\otimes\zeta^*)V) \\
    &= s_U^*(1\otimes (1\otimes\zeta^*)V) s_U = \zeta^* (s_U^*\otimes 1)(1\otimes V)s_U,
  \end{align*}
  using the left and right quantum traces from the Preliminaries. But
  $(s_U^*\otimes 1)(1\otimes V)s_U$ is an intertwiner from $1$ to $x$, and since
  $1\notin F$ it must vanish. As a result we have
  $b_{0,13} (Q_U^{-2}\otimes p_x\otimes \pi_{U\bar u})(\id\otimes\beta)(b_0) = 0$ and
  $(1\otimes p_x\otimes 1) b_{13}(\id\otimes\beta)(b) = 0$. Since this holds for
  any $x\in F$ this shows that $\Sigma(b)/\|b\|$ fulfills the conditions in
  Definition~\ref{Def C*faithful} with $\eta = 0$ and $M_k(\C)\simeq B(H_U)$.
\end{proof}

We come back to the general case. The following reformulation of (strong) $C^*$-faithfulness is crucial for the application to $C^*$-simplicity.

\begin{prop}\label{Prop Charact C*faithul}
  The action $\G\acts A$ is $C^*$-faithful (resp.\ strongly $C^*$-faithful) {\bf iff} for every $\eta > 0$ and every minimal (resp.\ finite rank) central
  projection $p\in c_c(\G)$ with $\epsilon(p)=0$ there exists $k\in\N^*$ and $b\in (A\otimes M_k(\C))_+$ such that $\|b\| = 1$ and
  $\|p\otimes b + (p\otimes 1)(\alpha\otimes\id)(b)\|\leq 1+\eta$.
\end{prop}

The proposition follows immediately from the following two lemmas, which are probably known to experts, by taking $A = p\otimes b$ and
$B = (p\otimes 1)(\alpha\otimes\id)(b) \in c_0(\G)\otimes A \otimes M_k(\C)$.

\begin{lem}\label{Asymptotic Orthogonality Lemma}
  Fix $\epsilon \in \left]0,\frac12\right]$. Then there exists $n\in\N$ such that, for any positive operators $A$, $B\in B(H)$ satisfying $\|A\|\leq 1$,
  $\|B\|\leq 1$ and $\|A+B\|\leq 1+\epsilon$, we have $\|A^n B^n\| \leq 14\epsilon$.
\end{lem}

\begin{proof}[Proof]
  The assumption can also be written $0\leq A, B\leq 1$, $0\leq A+B \leq 1+\epsilon$.

  Step 1 (well-known). If $\zeta\in H$ satisfies $\|\zeta\| = 1$ and $\|A\zeta\|^2 \geq 1-\epsilon^2$, then $\|A\zeta-\zeta\|\leq\epsilon$.  Indeed, writing
  $\zeta = \zeta-A\zeta+A\zeta$ and noting that $A-A^2 \geq 0$ we can write
  \begin{displaymath}
    \epsilon^2 \geq \|\zeta\|^2 - \|A\zeta\|^2 =
    \|\zeta-A\zeta\|^2 +2\Re (\zeta-A\zeta \mid A\zeta) \geq\|\zeta-A\zeta\|^2.
  \end{displaymath}
  This implies also $(\zeta\mid A\zeta) \geq 1-\epsilon$.

  Step 2. Fix $\zeta$ as above and $\xi\in H$ such that $\|\xi\| = 1$, $\|B\xi\|^2\geq 1-\epsilon^2$. We have the following inequalities:
  $(\zeta\mid A\zeta)\geq 1-\epsilon$, $(\xi\mid B\xi)\geq 1-\epsilon$, $(\zeta\mid B\zeta)\geq 0$, $(\xi\mid A\xi)\geq 0$. Also,
  \begin{displaymath}
    (\zeta\mid A\xi) + (\xi\mid A\zeta)=
    2\Re(\xi\mid A\zeta) \geq  2\Re(\xi\mid\zeta)
    - 2\|\xi\|\|\zeta-A\zeta\| \geq 2\Re(\zeta\mid\xi)-2\epsilon.
  \end{displaymath}
  Similarly $(\xi\mid B\zeta) + (\zeta\mid B\xi) \geq 2\Re (\zeta\mid\xi)-2\epsilon$. Summing all these inequalities we get
  \begin{displaymath}
    (\zeta+\xi\mid (A+B)(\zeta+\xi)) \geq 2 + 4\Re(\zeta\mid\xi) - 6\epsilon.
  \end{displaymath}
  But on the other hand we have $(\zeta+\xi\mid (A+B)(\zeta+\xi)) \leq \|\zeta+\xi\|^2 \|A+B\| \leq (2+2\Re(\zeta\mid\xi))(1+\epsilon)$. Subtracting we obtain
  $\Re(\zeta\mid\xi) \leq 4\epsilon/(1-\epsilon) \leq 8\epsilon$, if $\epsilon\leq\frac 12$. We can apply this to $-\xi$ and we thus get
  $|\Re(\zeta\mid\xi)|\leq 8\epsilon$. We can also apply this to $i\xi$ and we get $|\Im(\zeta\mid\xi)|\leq 8\epsilon$, so that finally
  $|(\zeta\mid\xi)|\leq 8\sqrt 2\epsilon\leq 12\epsilon$.

  This can be reformulated as follows: denote $P$, $Q$ the spectral projections of $A$, $B$ for the interval $\mathopen[\sqrt{1-\epsilon^2},1\mathclose]$. Then for any $\zeta$, $\xi\in H$ such that $P\zeta\neq 0$, $Q\xi\neq 0$ we can apply the previous result to $\zeta' = P\zeta / \|P\zeta\|$, $\xi' = Q\xi / \|Q\xi\|$ and we get
  \begin{displaymath}
    |(P\zeta \mid Q\xi)| \leq 12\epsilon \|P\zeta\| \|Q\xi\| \leq 12\epsilon \|\zeta\| \|\xi\|.
  \end{displaymath}
  This holds as well if $P\zeta = 0$ or $Q\xi = 0$, so that we have proved $\|PQ\|\leq 12\epsilon$.

  Step 3. Take $n\in\N$ such that $(1-\epsilon^2)^n \leq \epsilon^2$. For any unit vectors $\zeta$, $\xi\in H$ we have then
  $\|A^n(1-P)\zeta\|\leq (1-\epsilon^2)^{n/2} \|(1-P)\zeta\| \leq \epsilon$ and similarly $\|B^n(1-Q)\xi\|\leq \epsilon$. We can then write
  \begin{align*}
    |(A^n\zeta \mid B^n\xi)| &\leq |(A^n(1-P)\zeta \mid B^n\xi)|
                               + |(A^nP\zeta \mid B^n(1-Q)\xi)| + |(A^nP\zeta \mid B^nQ\xi)| \\
                             &\leq \epsilon + \epsilon + |(A^n\zeta \mid PQB^n\xi)| \leq 14\epsilon.
  \end{align*}
  This proves that $\|A^n B^n\|\leq 14\epsilon$.
\end{proof}

The second lemma is much easier and certainly well known as well.

\begin{lem}\label{Easy orthogonality lemma}
  Fix $\epsilon\in \left[0,+\infty\right[$.  Then for any positive operators $A$, $B\in B(H)$ such that $\|A\|=\|B\|=1$ and $\|AB\|\leq\epsilon$, we have
  $\|A+B\|\leq 1+2\epsilon$.
\end{lem}

\begin{proof}
  Since $A+B\geq A$ the assumption entails $\|A+B\|\geq 1$. Also, $0\leq A$, $B\leq 1$ implies $0\leq A^2+B^2 \leq A+B$. We can thus write
  \begin{align*}
    \|A+B\|^2 &= \|(A+B)^2\| = \|A^2+B^2+AB+BA\| \\ &\leq \|A^2+B^2\| + 2\epsilon
                                                      \leq \|A+B\|+2\epsilon\|A+B\|
  \end{align*}
  and the result follows.
\end{proof}

The following is a direct but important consequence of the previous reformulation of strong $C^*$-faithfulness, which does not rely on the multiplicative structure of the underlying algebra.

\begin{cor}\label{Cor C*faithful universal}
  If $\Phi : A \to B$ is $\G$-equivariant, unital and completely isometric, then (strong) $C^*$-faithfulness passes from $A$ to $B$. In particular, if a
  $\G$-boundary $A$ is (strongly) $C^*$-faithful, then so is the universal Furstenberg boundary $C(\partial_F\G)$.
\end{cor}

\begin{thm}\label{Strong Faithfulness Implies Faithful Expectations}
  Let $A$ be a $C^*$-algebra with strongly $C^*$-faithful action of $\G$, and let $E : A\rtimes_r\G \to A$ be a conditional expectation. Then the restriction
  of $E$ to $\Cred\G$ factors $h$.
\end{thm}

\begin{proof}
  Let $A$ be faithfully represented on a Hilbert space $H$. Recall that we identify $A$ and $\Cred\G$ as $C^*$-subalgebras of $A\rtimes_r\G$, in such a way that
  $(1\otimes a)W = W\alpha(a)\in M(c_0(\G)\otimes A\rtimes_r\G)$ for any $a\in A$. Denote $\Psi : \Cred\G\to B(H)$ the restriction of $E$, and $X = (\id\otimes \Psi)(W) \in M(c_0(\G)\otimes A)$. Applying $\id\otimes E$ to the previous relation we obtain $(1\otimes a)X = X\alpha(a)$.

  Take $\eta>0$ and $p\in c_c(\G)$ a finite-rank central projection such that $\epsilon(p) = 0$. Let $k\in \N$, $b\in A\otimes M_k(\C)$ be the elements given by
  Definition~\ref{Def C*faithful} for $\eta$ and $p$.  Since $\|b\| = 1$ we can find $\zeta\in H\otimes \C^k$ such that $\|\zeta\| \leq \sqrt 2$ and $\|b\zeta\| = 1$. Put
  $\xi = b\zeta$. We have
  \begin{align*}
    (p\otimes\omega_\xi)(X\otimes 1) &= (\id\otimes\omega_\zeta)[(p\otimes b)(X\otimes 1)(p\otimes b)]
    \\ &= (\id\otimes\omega_\zeta)[(X\otimes 1)(\alpha\otimes\id)(b)(p\otimes b)]
  \end{align*}
  hence $\|(p\otimes\omega_\xi)(X\otimes 1)\|\leq 2\eta$.
    
  Denoting $\Psi^k : A\rtimes_r\G \to B(H \otimes\C^k)$ the amplification $\Psi^k(a) = \Psi(a)\otimes\id$, and $L(\mu) = (\mu\otimes\id)(W)\in \Cred\G$ for
  $\mu\in\ell^1(\G)$, this implies $|\omega_\xi(\Psi^k(L(\mu)))| \leq 2\eta\|\mu\|_1$ if $\mu$ is supported on $p$. Note that $h(L(\mu)) = \mu(p_0)$, where $p_0$
  is the support of $\epsilon$, and $L(\epsilon) = 1$. Hence if now $\mu$ is supported on $p_0+p$ we have
  $|(\omega_\xi\circ \Psi^k - h)(L(\mu))|\leq 2\eta\|\mu\|_1$.  Writing $\xi = \sum_{i=1}^k \xi_i\otimes e_i$ and denoting $\varphi\in S(A)$ the sum of the
  restrictions of the forms $\omega_{\xi_i}$ to $A$, we have $\omega_\xi(\Psi^k(L(\mu))) = \varphi(\Psi(L(\mu)))$. Thus by letting $\eta\to 0$ we see that
  $h(L(\mu))$ belongs to the closure of $\{\varphi(\Psi(L(\mu))) : \varphi\in S(A)\}$.
  Finally since all elements $x\in \C[\G]$ can be written $x = L(\mu)$ as above, Lemma~\ref{Weak* Approximately Equal Lemma} shows that $\Psi$ factors $h$.
\end{proof}

\begin{cor}\label{Strongly C* faithful Implies C*simple and Faithful}
  If $\G$ admits a $\G$-boundary with a strongly $C^*$-faithful action, then $\G$ has the \PAP{}. In particular, $\G$ is $C^*$-simple.
\end{cor}

\begin{proof}
  By Corollary~\ref{Cor C*faithful universal}, the action of $\G$ on $C(\partial_F\G)$ is strongly $C^*$-faithful. Let $\Psi : \Cred\G \to C(\partial_F\G)$ be a
  $\G$-equivariant ucp map. By $\G$-injectivity, extend it to $E : C(\partial_F\G)\rtimes_r\G \to C(\partial_F\G)$. By $\G$-rigidity, $E$ is a conditional
  expectation. We can thus apply Theorem~\ref{Strong Faithfulness Implies Faithful Expectations} which shows that $\Psi$ factors $h$, and by Corollary \ref{PAPUniqueGMap}, $\G$ has the \PAP{}. By Proposition \ref{PowersAveSimplicity}, $\G$ is $C^*$-simple.
\end{proof}

\section{The boundary of $\F U_F$ is a $\G$-boundary}

In this section we prove that $C(\partial_G\F U_F)$ is an $\F U_F$-boundary, using the unique stationarity method from \cite{KKSV22}. We use the notation from Section~\ref{Section Boundary FUF}. The main object of the proof is the canonical ``harmonic state'' on $C(\partial\F U_F)$ that we describe now.

Denote $\qTr_n = \sum_{|x|=n} \qTr_x$,
$\qdim(n) = \sum_{|x|=n}\qdim(x)$ and $\qtr_n = \qTr_n(\cdot)/\qdim(n)$ ---
which is different from $\sum_{|x|=n}\qtr_x$. 
We have the following convolution formula in $\ell^1(\F U_F)$:
\begin{equation}\label{eq_convolution_traces}
  \qTr_x * \qTr_y = {\ts\sum_z} m(z,x\otimes y) \qTr_z,
\end{equation}
where $m(z,x\otimes y) = \dim \Hom(z,x\otimes y)$.
According to this formula and to
the fusion rules of $\F U_F$ we have
$\qTr_1 * \qTr_n = \qTr_{n+1} + 2 \qTr_{n-1}$ for $n\geq 1$. Since
$\qTr_1(1) = \qdim(u)+\qdim(\bar u) = 2[2]_q$, this gives in particular
$2[2]_q\qdim(n) = \qdim(n+1)+2\qdim(n-1)$. Denoting $\kappa$ the unique number
in $]0,1[$ such that $\kappa+\kappa^{-1} = \sqrt 2[2]_q$, this yields
$\qdim(n) = \sqrt 2^n[n+1]_\kappa$.

More generally, if $\mu$ is a probability measure on $I$ we denote
$\psi_\mu = \sum \mu(x)\qtr_x$ the associated state on $c_0(\G)$. It is
shown in \cite{VVV10} that if $\mu$ is generating and finitely supported, then
the sequence $\psi_\mu^{*n}$ converges $*$-weakly on $B = C(\beta_G\F U_F)\subset \ell^\infty(\F U_F)$ to a state
$\omega\in B^*$, which is KMS with respect to the modular automorphism group and  factors to a faithful state still denoted $\omega$ on
$B_\infty$. Moreover the ``$\mu$-harmonic state'' $\omega$ is
$\psi_\mu$-stationary, meaning that $\psi_\mu*\omega = \omega$, and the
corresponding Poisson map $P_\mu : C(\partial_G\F U_F) \to H^\infty(\F U_F,\mu)$
is completely isometric.

Now we restrict to the case $\mu = \mu_1 := \frac 12(\delta_u + \delta_{\bar u})$,
i.e.\ $\psi_\mu = \qtr_1$, which is sufficient to prove that
$C(\partial_G\F U_F)$ is an $\F U_F$-boundary. From the remarks above it follows that
$\qtr_1^{*n}$ is a convex combination of the states $\qtr_k$, $0\leq k\leq n$,
$k+n$ even. We show below that already the sequence $(\qtr_n)_n$ converges
$*$-weakly on $B$ to a state $\omega \in B^*$, which then coincides clearly with
the harmonic state of the previous paragraph. We also compute $\omega$
explicitly, although we will not need it in the sequel. Recall the notation $\alpha^{(l)} = \alpha\bar\alpha\alpha\cdots$ from Section~\ref{Section Boundary FUF}, for $\alpha = u$ or $\bar u$.

\begin{lem} \label{lem_harmonic_state} Fix $x\in I$, $x\neq 1$, and
  $a\in p_x \ell^\infty(\F U_F) \simeq B(H_x)$. Then
  $\qtr_n(\psi_{x,\infty}(a))$ converges as $n\to\infty$. Moreover, let
  $l\in\N^*$ be the unique integer such that $x\simeq x'\otimes \alpha^{(l)}$
  with $x'\in I$, $\alpha \in \{u,\bar u\}$. Then the limit is:
  \begin{gather*}
    \lim_{n\infty}\qtr_n(\psi_{x,\infty}(a)) = \omega(\psi_{x,\infty}(a)) = \qtr_x(a)\times \omega_I(\partial I(x)), \quad \text{where} \\
    \omega_I(\partial I(x)) = \qdim(x) \left({\ts\frac{\kappa}{\sqrt 2}}\right)^{|x|} \left( 1 - {\ts\frac{\kappa}{\sqrt 2}\frac {[l]_q}{[l+1]_q}}\right).
  \end{gather*}
\end{lem}

\begin{proof}
  If $z\subset x\otimes y$ with $|z| = |x|+|y|$ we have $\qTr_z = \qTr_x\otimes\qTr_y$ on $B(H_z) \subset B(H_x)\otimes B(H_y)$. Observe moreover that
  $(\id\otimes \qTr_y)(P(z,x\otimes y)) = (\qdim(z)/\qdim(x))\id_x$ since this element of $B(H_x)$ is an intertwiner, hence a multiple of $\id_x$. From this we get
  \begin{displaymath}
    \qTr_z(\psi_{x,z}(a)) = (\qTr_x\otimes\qTr_y)[P(z,x\otimes y)
      (a\otimes\id_y)] = {\ts\frac {\qdim(z)}{\qdim(x)}} \qTr_x(a).
  \end{displaymath}
  Summing over $z \in I_n$, with $n\geq |x|$, we get
  $\qTr_n(\psi_{x,\infty}(a)) = \qdim(x,n) \qtr_x(a) $, where we put
  $\qdim(x,n) = \sum \{\qdim(z) \mid |z| = n, z\geq x\}$. Thus it remains to
  show that $\qdim(x,n) / \qdim(n)$ has the limit given in the statement when
  $n\to\infty$. According to the fusion rules the sequence $(\qdim(x,n))_n$
  satisfies the same recursion equation than $(\qdim(n))_n$, but with a
  different initialization. This already shows the existence of the limit.

  To compute the limit, first note that $\qdim(x,n) = \qdim(x')\qdim(\alpha^{(l)}, n-|x'|)$. Moreover
    $\qdim(n-|x'|)/\qdim(n) \to (\kappa/\sqrt 2)^{|x'|}$ as $n\to\infty$, thanks to the formula $\qdim(n) = \sqrt 2^n[n+1]_\kappa$. As a result it suffices to
    consider the case $x = \alpha^{(l)}$. In this case we obtain the result by solving the above recursion relation with the
  initialization $\qdim(x,l) = \qdim(x) = [l+1]_q$ and $\qdim(x,l+1) = \qdim(xu) + \qdim(x\bar u) = [2]_q[l+1]_q + [l+2]_q$.
\end{proof}

The result above completely determines the harmonic state $\omega$ on
$C(\partial_G\F U_F)$ and in particular the classical probability measure
$\omega_I$ on $\beta I$. We get in particular
$\omega_I(\partial I(x)) \leq ([2]_q\kappa/\sqrt 2)^{|x|}$. One can check that
$[2]_q\kappa/\sqrt 2 < 1$, and thus $\omega_I$ is non-atomic. Note that for arbitrary $\mu$ it is already shown in \cite{VVV10} that $\omega_I$ has no atoms at
points of the form $x\alpha^{(\infty)}$, and in \cite{VVthesis} that $\omega_I$ has no atoms at all. The following Lemma can be seen as a
quantitative and uniform version of the fact that $\omega_I$ has no atoms at
points of the form $x\alpha^{(\infty)}$, in the case $\mu = \mu_1$: indeed by taking the limit $n\to\infty$ it yields $\omega_I(\partial I(x\alpha^{(k+1)}))\leq 2^{-k}$ for any $x\in I$, $k\in\N$.

\begin{lem}\label{lem_trace_mod}
  For fixed $p\leq p+k\leq n$, denote $I_n^{(p,k)} \subset I_n$ the set of words $y\in I_n$ which have a
  repetition of letters between the $p$th and $(p+k)$th letter, i.e.\ we have
  $y\simeq y'\otimes y''$ with $p \leq |y'|\leq p+k$. Denote
  $\qtr_n^{(p,k)} = \qdim(n)^{-1} \sum \{\qTr_y \mid y\in I_n^{(p,k)}\}$. Then for all $n\geq p+k$ we have
  \begin{displaymath}
    \|\qtr_n - \qtr_n^{(p,k)}\|\leq 2^{-k}.
  \end{displaymath}
\end{lem}

\begin{proof}
  We have
  $\qdim(n) \|\qtr_n - \qtr_n^{(p,k)}\| = \sum \{\qTr_y(1) \mid y\in
  {}^cI_n^{(p,k)}\}$. Any $y\in {}^cI_n^{(p,k)} = I_n\setminus I_n^{(p,k)}$ can
  be written $y = y_1\alpha^{(k+1)}y_2$ with $|y_1| = p-1$, $\alpha = u$ or
  $\bar u$. Observe moreover that
  $\qdim(y_1\alpha^{(k+1)}y_2) \leq \qdim(y_1zy_2)$ for any $z\in I_{k+1}$,
  hence $\qTr_y(1) = \qdim(y) \leq 2^{-k-1} \sum_{z\in
    I_{k+1}}\qdim(y_1zy_2)$. Summing over $y_1\in I_{p-1}$,
  $\alpha\in\{u,\bar u\}$ and $y_2 \in I_{n-p-k}$ this yields
  $\qdim(n) \|\qtr_n - \qtr_n^{(p,k)}\| \leq 2^{-k}
  \sum_{y_1,z,y_2}\qdim(y_1zy_2) = 2^{-k}\qdim(n)$.
\end{proof}

\bigskip

We prove now that $\partial_G\F U_F$ is indeed a $\F U_F$-boundary, using the
unique stationarity method from \cite{KKSV22}. We
still work with the harmonic state $\omega$ on $C(\partial_G \F U_F)$ induced by
$\psi_\mu = \qtr_1 \in \ell^\infty(\F U_F)_*$. Recall that $\omega$ is
$\qtr_1$-stationary, meaning that ${\qtr_1} * \omega = \omega$, and the
corresponding Poisson map
$P_1 = (\id\otimes\omega)\beta : C(\partial_G\F U_F) \to H^\infty(\F U_F,\qtr_1)
\subset \ell^\infty(\F U_F)$ is completely isometric \cite{VVV10}.  In this
setting it follows that $\partial_G\F U_F$ is an $\F U_F$-boundary if one can
prove that $\omega$ is the {\em unique} $\qtr_1$-stationary state $\nu$ on
$C(\partial_G\F U_F)$ \cite{KKSV22}. 

This unique stationarity result was established in \cite{HHN22} under the
additional assumption that $\nu$ is invariant with respect to the adjoint action
of the dual compact quandum group --- in other words $\omega$ is the unique
stationary state with respect to the natural action of the Drinfeld double
$D(\F U_F)$ on $\partial_G \F U_F$. Considering the action of the Drinfeld
double makes the situation much more rigid: for instance in the orthogonal
case, $\partial_G \F O_F$ is the {\em only} non-trivial $D(\F O_F)$-boundary (see Example 3.19 of \cite{HHN22}).

\bigskip

We start with the intermediate goal of proving that a $\qtr_1$-stationary state
$\nu$ on $C(\partial_G \F U_F)$ has no atoms at classical points of the form
$x\alpha^{(\infty)}\in\partial I$. If $\F U_F$ is unimodular, i.e.\ $F$ is a
multiple of a unitary matrix, this is easy. Indeed in this case, the left
and right quantum traces are equal and in particular
$\ell^\infty(I)\subset \ell^\infty(\F U_F)$ is stable under left convolution by
$\qtr_1$, which then corresponds to the Markov operator of a classical random
walk on $I$. But Proposition 3.23 of \cite{HHN22} shows that this classical
random walk has a unique stationary probability measure, so that necessarily
$\nu_I = \omega_I$, which is already known to have no atoms at points
$x\alpha^{(\infty)}$, as recalled above.

The general case is more involved but still follows the classical strategy: if
$\nu$ is a $\mu$-stationary measure on the Gromov boundary $\partial_G F_N$ of
a classical free group, the function $h : x\to \nu(\{x\})$ on $\partial_G F_N$
is summable and $\mu$-harmonic with respect to the action of $F_N$ on its
boundary. Since all orbits of the action are infinite, this implies $h=0$. One
can also work with the associated functions $\gamma\mapsto h(\gamma\cdot x_0)$
on $F_N$, which are easier to manipulate in the quantum case --- this is the ``classical analogue'' for the elements $h$, $\bar h$ introduced below.

\bigskip

Fix a $\qtr_1$-stationary state $\nu$ on $C(\partial_G \F U_F)$ ; we still
denote $\nu$ the induced state on
$C(\beta_G \F U_F) \subset \ell^\infty(\F U_F)$. Note that
$\pi_{u^{(k)}} = \sum_{y\in I} p_{u^{(k)}y}$, $k\in\N$, is a decreasing sequence of
projections in $C(\beta_G\F U_F)$. Denote
$h^{(k)} = P_\nu(\pi_{u^{(k)}}) = (\id\otimes\nu)\Delta(\pi_{u^{(k)}})$ and
consider the decreasing limit
$h = \lim_{k\infty} h^{(k)} \in \ell^\infty(\F U_F)$. Similarly, denote
$\bar h^{(k)} = P_\nu(\pi_{\bar u^{(k)}})$ and
$\bar h= \lim_{k\infty} \bar h^{(k)}$. Our goal is to prove that $h=\bar h =0$.

\begin{lem} \label{lem_transform_stabilizer}
  Consider the element $h = (h_y)_{y\in I}\in \ell^\infty(\F U_F)$ associated
  to a state $\nu \in C(\partial_G\F U_F)^*$ as above.  For $x\in I(\bar u)$ and $l\in\N$ we have
  $h_{u^{(2l)}x} = p_{u^{(2l)}}\otimes h_x$ in the identification
  $B(H_{u^{(2l)}x})\simeq B(H_{u^{(2l)}})\otimes B(H_x)$. Similarly for
  $x\in I(u)$ we have $h_{u^{(2l+1)}x} = p_{u^{(2l+1)}}\otimes \bar h_x$.
\end{lem}

\begin{proof}
  Take $k> 2l$. Start with $x = 1$. We have
  $(p_{u^{(2l)}}\otimes p_y)\Delta(p_{u^{(k)}z}) \neq 0$ iff
  $u^{(k)}z \subset u^{(2l)}\otimes y$ iff $y\subset u^{(2l)} \otimes u^{(k)}z$
  iff $y = u^{(k+2m)}z$ with $|m|\leq l$. In particular
  $(p_{u^{(2l)}}\otimes p_y)\Delta(\pi_{u^{(k)}}) \neq 0$ $\Rightarrow$
  $y\in I(u^{(k-2l)})$. On the other hand for
  $y \in I(u^{(k+2l)})$, all subobjects of $u^{(2l)}\otimes y$ are in
  $I(u^{(k)})$. As a result we have
  \begin{equation} \label{eq_bounds_pi}
    p_{u^{(2l)}}\otimes\pi_{u^{(k+2l)}} \leq (p_{u^{(2l)}}\otimes 1)\Delta(\pi_{u^{(k)}})
    \leq p_{u^{(2l)}}\otimes\pi_{u^{(k-2l)}}.
  \end{equation}

  Coming back to a general $x\in I(\bar u)$ we have as well
  $(p_{u^{(2l)}x}\otimes p_y)\Delta(p_{u^{(k)}z}) \neq 0$ iff
  $u^{(k)}z \subset u^{(2l)}x\otimes y$ iff
  $y\subset \bar x u^{(2l)} \otimes u^{(k)}z$ iff $y = \bar xu^{(k+2m)}z$ with
  $|m|\leq l$. Moreover in this case we have
  $u^{(2l)}x\otimes y = u^{(2l)}\otimes x\otimes \bar x\otimes u^{(k+2m)}z =
  \bigoplus_{t\subset x\otimes\bar x} u^{(2l)}\otimes t\otimes u^{(k+2m)}z$, and
  the objects $t$ appearing in the sum start with $\bar u$, except $t = 1$. As a
  result the isometric morphism $V(u^{(k)}z,u^{(2l)}x\otimes y)$ has to factor as
  $(\id\otimes V({u^{(k+2m)}z,}$ $x\otimes\bar
  xu^{(k+2m)}z))V(u^{(k)}z,u^{(2l)}\otimes u^{(k+2m)}z)$. Recalling from Section~\ref{Section Boundary FUF} that one can use such morphisms to compute the coproduct, we obtain for
  $y = \bar xu^{(k+2m)}z$:
  \begin{align*}
    (p_{u^{(2l)}x}\otimes p_y)\Delta(p_{u^{(k)}z})
    &= (p_{u^{(2l)}}\otimes p_x\otimes p_y)
      (\id\otimes\Delta)[(p_{u^{(2l)}}\otimes p_{u^{(k+2m)}z})\Delta(p_{u^{(k)}z})] \\
    &= (p_{u^{(2l)}}\otimes p_x\otimes p_y)
      (\id\otimes\Delta)[(p_{u^{(2l)}}\otimes 1)\Delta(p_{u^{(k)}z})],
  \end{align*}
  since $z$ and $m$ are uniquely determined by $y$. Here we are using the identifications $p_{u^{(2l)}x}\ell^\infty(\F U_F) \simeq B(H_{u^{(2l)}x}) \simeq B(H_{u^{(2l)}})\otimes B(H_x) \simeq p_{u^{(2l)}}\ell^\infty(\F U_F)\otimes p_x\ell^\infty(\F U_F)$. Summing over $z$ (first) and $y$ gives  
  \begin{displaymath}
    (p_{u^{(2l)}x}\otimes 1)\Delta(\pi_{u^{(k)}}) = (p_{u^{(2l)}}\otimes p_x\otimes 1)
    (\id\otimes\Delta)[(p_{u^{(2l)}}\otimes 1)\Delta(\pi_{u^{(k)}})].
  \end{displaymath}
  
  Now we use~\eqref{eq_bounds_pi} which yields:
  \begin{displaymath}
    p_{u^{(2l)}}\otimes [(p_x\otimes 1)\Delta(\pi_{u^{(k+2l)}})] \leq
    (p_{u^{(2l)}x}\otimes 1)\Delta(\pi_{u^{(k)}}) \leq
    p_{u^{(2l)}}\otimes [(p_x\otimes 1)\Delta(\pi_{u^{(k-2l)}})].
  \end{displaymath}
  Finally we apply $(\id\otimes\id\otimes\nu)$ to obtain
  $p_{u^{(2l)}}\otimes h^{(k+2l)}_x\leq h^{(k)}_{u^{(2l)}x} \leq
  p_{u^{(2l)}}\otimes h^{(k-2l)}_x$. Letting $k\to\infty$ yields
  $h_{u^{(2l)}x} = p_{u^{(2l)}}\otimes h_x$.

  The odd case works similarly. We have $(p_{u^{(2l+1)}x}\otimes p_y)\Delta(p_{u^{(k)}z}) \neq 0$ iff
  $u^{(k)}z \subset u^{(2l+1)}x\otimes y$ iff
  $y\subset \bar x \bar u^{(2l+1)} \otimes u^{(k)}z$ iff $y = \bar x\bar u^{(k+2m+1)}z$ with
  $-l-1\leq m\leq l$. When $x=1$ we obtain
  \begin{displaymath} 
    p_{u^{(2l+1)}}\otimes\pi_{\bar u^{(k+2l+1)}} \leq
    (p_{u^{(2l+1)}}\otimes 1)\Delta(\pi_{u^{(k)}}) \leq
    p_{u^{(2l+1)}}\otimes\pi_{\bar u^{(k-2l-1)}}.
  \end{displaymath}
  In general if $y = \bar x\bar u^{(k+2m+1)}z$ we obtain successively:
  \begin{gather*}
    (p_{u^{(2l+1)}x}\otimes p_y)\Delta(p_{u^{(k)}z})
    = (p_{u^{(2l+1)}}\otimes p_x\otimes p_y)
      (\id\otimes\Delta)[(p_{u^{(2l+1)}}\otimes p_{\bar u^{(k+2m+1)}z})\Delta(p_{u^{(k)}z})], \\
      (p_{u^{(2l+1)}x}\otimes 1)\Delta(\pi_{u^{(k)}})
      = (p_{u^{(2l+1)}}\otimes p_x\otimes 1)
      (\id\otimes\Delta)[(p_{u^{(2l+1)}}\otimes 1)\Delta(\pi_{u^{(k)}})] \quad \text{and} \\
    p_{u^{(2l+1)}}\otimes [(p_x\otimes 1)\Delta(\pi_{\bar u^{(k+2l+1)}})] \leq
    (p_{u^{(2l+1)}x}\otimes 1)\Delta(\pi_{u^{(k)}}) \leq \hspace{3cm} \\ \hspace{6cm}
    \leq p_{u^{(2l+1)}}\otimes [(p_x\otimes 1)\Delta(\pi_{\bar u^{(k-2l-1)}})].
  \end{gather*}
  Applying $(\id\otimes\id\otimes\nu)$ this yields
  $p_{u^{(2l+1)}}\otimes \bar h^{(k+2l+1)}_x \leq h^{(k)}_{u^{(2l+1)}x} \leq
  p_{u^{(2l+1)}}\otimes \bar h^{(k-2l-1)}_x$ and the result follows by taking
  the limit $k\to\infty$.
\end{proof}

\begin{prop} \label{prop_no_atoms}
  Let $\nu$ be a $\qtr_1$-stationary state on $C(\partial_G\F U_F)$. Then
  $\nu_I$ has no atoms at points of the form $x\alpha^{(\infty)}$, i.e.\ we have
  $\lim_{k\infty} \nu(\pi_{x\alpha^{(k)}}) = 0$ for all $x\in I$.
\end{prop}

\begin{proof}
  From the stationarity property $(\qtr_1\otimes\nu)\Delta = \nu$ one easily
  obtains harmonicity of the elements $h$, $\bar h$ considered in
  Lemma~\ref{lem_transform_stabilizer}: we have
  $(\id\otimes\qtr_1)\Delta(h) = h$ and
  $(\id\otimes\qtr_1)\Delta(\bar h) = \bar h$. We have
  $(p_x\otimes\qtr_1)\Delta(p_z)\neq 0$ iff $z\subset x\otimes(u\oplus\bar u)$,
  hence the harmonicity equation for $h$ can be written
  \begin{equation}\label{eq_harmonic}
    h_x = {\ts\frac 12}(p_x\otimes \qtr_\alpha)\Delta(h_{x\alpha})
    + {\ts\frac 12} (p_x\otimes \qtr_{\bar\alpha})\Delta(h_{x\bar\alpha})
    + {\ts\frac 12} (p_x\otimes \qtr_{\bar\alpha})\Delta(h_{x'}).
  \end{equation}
  if $x = x'\alpha$ with $\alpha\in\{u,\bar u\}$ (we put $h_{x'} = 0$ if $x = 1$).
  
  Consider the function $\varphi : I \to \R_+$, $x \mapsto \qtr_x(h_x)$. Since
  quantum traces are compatible with subobjects we have, for
  $z\subset x\otimes y$ and $a_z\in B(H_z)$:
  $(\qTr_x\otimes \qTr_y)\Delta(a_z) = \qTr_z(a_z)$. Thus, applying $\qtr_x$
  to~\eqref{eq_harmonic} we obtain
  \begin{displaymath}
    \varphi(x) = {\ts\frac 12} p(x,x\alpha)\varphi(x\alpha)
    + {\ts\frac 12} p(x,x\bar\alpha) \varphi(x\bar\alpha)
    + {\ts\frac 12} p(x,x') \varphi(x'),
  \end{displaymath}
  where $p(x,y) = \qdim(y)/\qdim(x)\qdim(u)$ if $x$, $y$ are neighbours in the
  classical Cayley graph of $\F U_F$, i.e.\ $y\subset x\otimes (u\oplus\bar
  u)$. Note that with our notation $x\alpha = x\otimes\alpha$ so that
  $p(x,x\alpha) = 1$, and on the other hand
  $x\bar\alpha\oplus x' \simeq x\otimes\bar\alpha$ so that
  $p(x,x\bar\alpha) + p(x,x') = 1$. In short, $\varphi$ is harmonic with respect
  to the classical random walk on $I$ with transition probabilities $p(x,y)$.

  Besides, we claim that $\sum_{x\in I^+(\bar u)}\varphi(x)\leq 1$, where
  $I^+(\bar u) = I(\bar u)\cup\{1\}$. To see this, take a finite subset $F\subset I^+(\bar u)$ and put $k = 1+\max_{x\in F}|x|$, so that the subsets $I(\bar xu^{(k)})\subset \partial I$ are pairwise disjoint when $x\in F$.
  Consider the positive elements
  \begin{displaymath}
    \eta_x = (\qtr_x\otimes\id)\Delta(\pi_{u^{(k)}}) \in C(\beta_G\F U_F).
  \end{displaymath}
  We have $p_y\eta_x \neq 0$ iff there exists $z\in I$ such that
  $u^{(k)}z \subset x\otimes y$, equivalently,
  $y\subset \bar x \otimes u^{(k)}z$. If $x\in I^+(\bar u)$ we have
  $\bar x \otimes u^{(k)}z = \bar x u^{(k)}z$, hence we get
  $\eta_x = \pi_{\bar xu^{(k)}}\eta_x$. The projections $\pi_{\bar xu^{(k)}}$ are pairwise orthogonal when $x\in F$,
  hence we have
  $\| \ts \sum_{x\in F} \eta_x \| = \max_{x\in F} \|\eta_x\| \leq 1$. Since
  $\nu$ is a state we can then write
  \begin{displaymath}
      1\geq \nu\left({\ts\sum_{x\in F}\eta_x}\right) = \sum_{x\in F}\qtr_x(h_x^{(k)}) \geq \sum_{x\in F}\qtr_x(h_x) = \sum_{x\in F}\varphi(x).
  \end{displaymath}
  This holds for any finite $F\subset I^+(\bar u)$, hence he claim follows.
 Working with the elements
  $\bar\eta_x = (\qtr_x\otimes\id)$ $\Delta(\pi_{\bar u^{(k)}})\in C(\beta_G\F
  U_F)$ we see similarly that $\sum_{x\in I^+(u)} \bar\varphi(x)\leq 1$, where
  $\bar\varphi(x) = \qtr_x(\bar h_x)$.

  In particular $\varphi$ attains a maximum $M$ on $I^+(\bar u)$ and
  $\bar\varphi$ attains a maximum $\bar M$ on $I^+(u)$. For
  $y \in I\setminus I^+(\bar u)$ we can write $y = u^{(2l)}x$ with
  $x\in I(\bar u)$ or $y = u^{(2l+1)}x$ with $x\in I(u)$. In the first case we
  apply $\qtr_{u^{(2l)}x} = \qtr_{u^{(2l)}}\otimes\qtr_x$ to the identity of
  Lemma~\ref{lem_transform_stabilizer}, which yields
  $\varphi(u^{(2l)}x) = \varphi(x) \leq M$. In the second case we have similarly
  $\varphi(u^{(2l+1)}x) = \bar\varphi(x)\leq \bar M$. Altogether we see that
  $\varphi$ attains the maximum $\max(M,\bar M)$ on $I$. Since it is harmonic,
  it has to be constant. Since it is summable on $I^+(\bar u)$, which is
  infinite, we have in fact $\varphi = 0$. Hence $h = 0$. Similarly $\bar\varphi = 0$ and $\bar h = 0$.

  In particular we already have $\epsilon(h) = \lim_{k\infty}\nu(\pi_{u^{(k)}}) = 0$. Consider more generally a point $x\alpha^{(\infty)} \in \partial
  I$. Without loss of generality, one can assume that the last letter of $x$ is $\alpha$, i.e.\ $\bar x\in I(\bar\alpha)$, so that
  $x\alpha^{(k)}z = x\otimes\alpha^{(k)}z$ for all $k\in\N^*$, $z\in I$. We have then $(p_{\bar x}\otimes p_y)\Delta(p_{\alpha^{(k)}z}) \neq 0$ iff
  $y = x\alpha^{(k)}z$ and in that case one can take $V(\alpha^{(k)}z,\bar x\otimes y) = (\qdim \bar x)^{-1/2} (t_{\bar x}\otimes\id_{\alpha^{(k)}z})$ so that
  \begin{displaymath}
    (p_{\bar x}\otimes 1)\Delta(p_{\alpha^{(k)}z}) = (\qdim\bar x)^{-1} (t_{\bar x}t_{\bar x}^* \otimes p_{\alpha^{(k)}z}),
  \end{displaymath}
  where we identify $H_{x\alpha^{(k)}z} \simeq H_x\otimes H_{\alpha^{(k)}z}$. We apply $\qtr'_{\bar x}\otimes\id_x\otimes\id_{\alpha^{(k)}z}$ to this identity,
  where $\qtr'_{\bar x} = t_x^*(\id\otimes\,\cdot\,)t_x$ is the {\em right} quantum trace. According to the conjugate equations, this yields
  $(\qtr'_{\bar x}\otimes\id)\Delta$ $(p_{\alpha^{(k)}z}) = (\qdim\bar x)^{-1}(p_x\otimes p_{\alpha^{(k)}z}) = (\qdim\bar x)^{-1}p_{x\alpha^{(k)}z}$.  Summing over
  $z$ we obtain $(\qtr'_{\bar x}\otimes 1)\Delta(\pi_{\alpha^{(k)}}) = (\qdim\bar x)^{-1} \pi_{x\alpha^{(k)}}$. Finally we can apply $\nu$ and take the limit
  $k\to\infty$ to get
  \begin{displaymath}
    \lim_{k\infty}\nu(\pi_{x\alpha^{(k)}}) = (\qdim\bar x)\lim_{k\infty}\qtr'_{\bar x}(h_{\bar x}^{(k)}) = (\qdim\bar x)\qtr'_{\bar x}(h_{\bar x}) = 0,
  \end{displaymath}
  if $\alpha = u$. If $\alpha = \bar u$ the same argument holds with $h$ replaced by $\bar h$.
\end{proof}

\bigskip

We can now prove our main unique stationarity result. The global strategy is the
same as in the orthogonal case, see \cite{KKSV22}, however the techniques are
very different. The analysis of $\partial \F O_F$ ultimately relies on the
geometry of the fusion rules in the Temperley-Lieb category. This geometry also
appears in the case of $\F U_F$, but only over the points of the form
$x\alpha^{(\infty)} \in\partial I$. Thanks to the non-atomicity result above,
we will be able to ignore these points, and thus our analysis will rely only on
the combinatorics of the free monoid over $u$, $\bar u$, which is more
elementary than the geometry of the Temperley-Lieb category.

\begin{thm} \label{Unique stationary state}
  If $N\geq 3$, the harmonic state $\omega$ is the only $\qtr_1$-stationary state on $C(\partial \F U_F)$.
\end{thm}

\begin{proof}
  Let $\nu$ be a $\qtr_1$-stationary state on $B_\infty$.  As in \cite{KKSV22} we
  shall prove that $\nu\leq\omega$, which implies $\nu = \omega$ because both
  maps are states. For this we fix $x_1\in I$, $a_{x_1}\in B(H_{x_1})_+$ with
  $\|a_{x_1}\|\leq 1$, and we put $a = \psi_{x_1,\infty}(a_{x_1})$. It suffices
  then to prove $\nu(a)\leq\omega(a)$. For any $z\in I$ we denote
  $a_z = p_z a = \psi_{x_1,z}(a_{x_1}) \in B(H_z)$ (which vanishes unless
  $z\geq x_1$).

  By Proposition~\ref{prop_no_atoms} the measure $\nu_I$ has no atoms at points of the form $x_1\alpha^{(\infty)}$ with $\alpha = u$ or $\bar u$, hence we have the
  decreasing limit $\nu(\pi_{x_1\alpha^{(k)}})) \to_k \nu_I(\{x_1\alpha^{(\infty)}\}) = 0$. Take $\epsilon > 0$ and choose $k$ such that
  $(|x_1|+k)2^{-k} \leq\epsilon$ and $\nu(\pi_{x_1\alpha^{(k)}}) \leq \epsilon$ for all $\alpha = u$, $\bar u$. The elements
    $a = \sum_{y\in I}\psi_{x_1,x_1 y}(a_{x_1})$ and
    $a' := \sum_{x_2\in I_k} \psi_{x_1x_2,\infty}(a_{x_1x_2}) = \sum_{x_2\in I_k, y\in I} \psi_{x_1,x_1x_2y}(a_{x_1})$ are equal from length $|x_1|+k$ on, meaning that $p_la = p_la'$ if $l\geq |x_1|+k$. Since $\nu$, seen as a state on $B$, vanishes on $c_0(\G)$, we obtain the decomposition $\nu(a) = \nu(a') = \sum_{x_2\in I_k}\nu(\psi_{x_1x_2,\infty}(a_{x_1x_2}))$. For $x_2 = \alpha^{(k)}$, $\alpha = u$, $\bar u$ and $x = x_1x_2$ we have
  $\nu(\psi_{x,\infty}(a_x)) \leq \epsilon$ since $\|\nu \psi_{x,\infty}\| = \nu(\pi_{x_1\alpha^{(k)}}) \leq\epsilon$. Hence
  $\nu(a) \leq \sum_{x_2\in I'_k}\nu(\psi_{x_1x_2,\infty}(a_{x_1x_2})) + 2\epsilon$, where $I'_k = I_k \setminus \{u^{(k)},\bar u^{(k)}\}$.

  Fix now $x = x_1x_2$ with $x_2\in I'_k$ and consider $\nu(\psi_{x,\infty}(a_x))$. By iterating the stationarity identity ${\qtr_1}*\nu = \nu$ we
  obtain ${\qtr_n}*\nu = \nu$ for any $n$. Applying Lemma~\ref{lem_boundary_action} with $k$ replaced by $|x_1|+k$ and $\epsilon$ replaced by $2^{-k}\epsilon$ we get $r$ such that
  \begin{equation}\label{eq unique stat proof}
    \nu(\psi_{x,\infty}(a_x)) = (\qtr_n\otimes\nu)\Delta(\psi_{x,\infty}(a)) \leq 2^{-k}\epsilon + {\ts\sum_{l=0}^n} (\qtr_n\otimes \nu\psi_{r,\infty})(b_{n+r-2l}),
  \end{equation}
  where $b_s = (p_n\otimes p_r)\Delta(a_s) = \sum (p_y\otimes p_z)\Delta(a_w)$, with the sum running over $y\in I_n$, $z\in I_r$ and $w\in I_s$, $w\geq x$ (so that $\psi_{|x|,s}(a_x) = \sum a_w$).

  \bigskip

  We start with the terms $l > n - |x|$ in~\eqref{eq unique stat proof}. For such an $l$ we shall show that the positive linear form
  $\varphi : a_x \mapsto (\qtr_n\otimes \nu\psi_{r,\infty})(b_{n+r-2l})$ has a small norm for large $n$. We first apply Lemma~\ref{lem_trace_mod} with $p = |x|$ and
  $k$ replaced by $2k$, getting $\|\varphi-\varphi'\|\leq 2^{-2k}$, where $\varphi'(a_x) = (\qtr_n^{(|x|,2k)}\otimes \nu\psi_{r,\infty})(b_{n+r-2l})$. We now prove a non-trivial estimate for $\|\varphi'\| = \varphi'(p_x) = (\qtr_n^{(|x|,2k)}\otimes \nu\psi_{r,\infty})(c_{n+r-2l})$, where $c_s = \sum_{x\leq w\in I_s}(p_n\otimes p_r)\Delta(p_w)$.
  
  Take $y\in I_n^{(|x|,2k)}$ and $z\in I_r$ such that
  $(p_y\otimes p_z)c_{n+r-2l} \neq 0$: then there exists a unique
  $w\in I_{n+r-2l}$ such that $w\geq x$ and $w\subset y\otimes z$, obtained by
  simplifying $l$ pairs of conjugate letters from the right of $y$ and the left
  of $z$. We have then
  $(p_y\otimes p_z)\Delta(p_w) = V(w,y\otimes z)V(w,y\otimes z)^*$. Since
  $n-l<|x|$ we can write $y \simeq y_1y_2\otimes y_3$ with $|y_1| = n-l$ and
  $|x|\leq |y_1y_2|\leq |x|+2k$. Then we must have
  $z \simeq \bar y_3\otimes \bar y_2z'$, because $y\otimes z$ contains $w$ of
  length $n+r-2l$, hence the last $l$ letters $y_2y_3$ of $y$ must simplify with
  the first $l$ letters of $z$. Thus we can take
  $V(w,y\otimes z) = (\id\otimes t_{y_3}\otimes \id)V(w,y_1y_2\otimes \bar
  y_2z')/\|t_{y_3}\|$. We have as well $\qTr_y = \qTr_{y_1y_2}\otimes\qTr_{y_3}$
  and we can perform the following computation:
  \begin{align*}
    &\|t_{y_3}\|^2\times (\qTr_y\otimes p_z)[V(w,y\otimes z)V(w,y\otimes z)^*] = \\
    & \hspace{3cm} = (\qTr_{y_1y_2}\otimes\qTr_{y_3}\otimes p_z)[(\id_{y_1y_2}\otimes t_{y_3}\otimes\id_{\bar y_2z'})
      P(w,y_1y_2\otimes\bar y_2z')(\id_{y_1y_2}\otimes t_{y_3}^*\otimes\id_{\bar y_2z'})] \\
    & \hspace{3cm} = (\qTr_{y_3}\otimes p_z)[(t_{y_3}\otimes\id_{\bar y_2z'})
      (\qTr_{y_1y_2}\otimes\id)[P(w,y_1y_2\otimes\bar y_2z')](t_{y_3}^*\otimes\id_{\bar y_2z'})] \\
    & \hspace{3cm} = Q_{\bar y_3}^{-2}\otimes (\qTr_{y_1y_2}\otimes\id)[P(w,y_1y_2\otimes\bar y_2z')]
      \leq \|Q_{\bar y_3}^{-2}\|\qdim(y_1y_2) p_z,
  \end{align*}
  since $\|\qTr_{y_1y_2}\| = \qdim(y_1y_2)$ and
  $(\qTr_{y_3}\otimes\id_{\bar y_3})(t_{y_3}t_{y_3}^*) = Q_{\bar y_3}^{-2}$.

  Observe moreover that $\|t_{y_3}\|^2 = \qdim(y_3) = \qdim(y)/\qdim(y_1y_2)$. Summing the previous estimate over $y\in I_n^{(|x|,2k)}$ and $z\in I_r$ we
  obtain, since $\nu\psi_{r,\infty}$ is a state:
  \begin{align*}
    \|\varphi'\|
    &= \qdim(n)^{-1}{\ts\sum_{y\in I_n^{(|x|,2k)}}}  (\qTr_y\otimes\nu\psi_{r,\infty})(c_{n+r-2l}) \\
    & \leq \qdim(n)^{-1} {\ts\sum_{y\in I_n^{(|x|,2k)}}} \frac{\|Q_{\bar y_3}^{-1}\|^2}{\qdim(y_3)^2}\qdim(y).
  \end{align*}
  According to Lemma~\ref{lem_woro_matrices} we have $\|Q_{\bar y_3}^{-1}\|/\qdim(y_3) \leq (q\rho)^{|y_3|} \leq (q\rho)^{n-|x|-2k}$. For the remaining sum
  we have by definition $\qdim(n)^{-1}\sum_y\qdim(y)\leq 1$, hence altogether $\|\varphi'\| \leq (q\rho)^{2n-2|x|-4k}$ and
  $\|\varphi\| \leq 2^{-2k} + (q\rho)^{2n-2|x|-4k}$. Recall now that $q\rho<1$ when $N\geq 3$. Letting $n\to \infty$ we thus obtain
  \begin{displaymath}
      \limsup_{n\infty} {\ts\sum_{l=n-|x|+1}^n} (\qtr_n\otimes \nu\psi_{r,\infty})(b_{n+r-2l}) \leq |x| 2^{-2k} \|a_x\|\leq |x| 2^{-2k},
  \end{displaymath}
  since there are $|x|$ terms in the sum.

  \bigskip
  
  Then we consider the terms $l\leq n-|x|$ in~\eqref{eq unique stat proof}. Take $y$, $z\in I$ with $|y| = n$, $|z| = r$. For each $s = n+r-2l$ there is at most one $w_l = xy_l'z_l' \geq x$
  of length $s$ such that $w_l\subset y\otimes z$, corresponding to a unique $v_l$ of length $l$ such that one can write $y = xy_l'\bar v_l$,
  $z=v_lz_l'$. Denote $p$ the maximal length of such a word $v_l$, with $p\leq n-|x|$. We have then $\bigoplus_{l=0}^p w_l \subset y\otimes z$ and
  \begin{displaymath}
    \sum_{l=0}^{n-|x|}(\qTr_y\otimes p_z)(b_s) = \sum_{l=0}^p(\qTr_y\otimes p_z)[V(w_l,y\otimes z)(a_x\otimes\id_{y'_lz'_l})V(w_l,y\otimes z)^*].
  \end{displaymath}

  Now recall that we have $x = x_1x_2$, with $x_2\in I'_k$ which can thus be written $x_2 = x_2'\otimes x_2''$. We have then
  $y = x_1x_2'\otimes x_2''y_l'\bar v_l$, $w_l = x_1 x_2'\otimes x_2'' y'_lz'_l$, so that $\qTr_y = \qTr_{x_1x_2'}\otimes \qTr_{x_2''y_l'}$ and we can take
  $V(w_l,y\otimes z) = \id_{x_1 x_2'}\otimes V(x_2'' y'_lz'_l, x_2''y_l'\bar v_l\otimes z)$. On the other hand we have by definition $a_x = \psi_{x_1,x}(a_{x_1}) = \psi_{x_1x'_2,x}(a_{x_1x'_2})$, and since $x=x_1x'_2\otimes x''_2$, we can take $V(x,x_1x'_2\otimes x''_2) = \id$ to compute $\psi_{x_1x'_2,x}$, so that finally
  $a_x = a_{x_1x'_2}\otimes \id_{x''_2}$. We have then
  \begin{align*}
    & (\qTr_y\otimes p_z)[V(w_l,y\otimes z)(a_x\otimes\id)V(w_l,y\otimes z)^*] = \\
    & \hspace{3cm} = \qTr_{x_1x_2'}(a_{x_1x'_2}) \times (\qTr_{x_2''y_l'}\otimes p_z)(P(x_2'' y'_lz'_l, x_2''y_l'\bar v_l\otimes z)) \\
    & \hspace{3cm} = \qTr_{x_1x_2'}(a_{x_1x'_2}) \frac{\qdim(x_2'' y'_lz'_l)}{\qdim(z)} \times p_z.
  \end{align*}
  Since $\bigoplus_{l=0}^p x_2'' y'_lz'_l \subset x_2''y_0'\otimes z$ we get the following upper bound for the sum over $l$:
  \begin{align*}
    \sum_{l=0}^{n-|x|}(\qTr_y\otimes p_z)(b_s) &\leq \qTr_{x_1x_2'}(a_{x_1x'_2}) \frac{\qdim(x_2'' y'_0\otimes z)}{\qdim(z)} \times p_z \\
    &= \qtr_{x_1x'_2}(a_{x_1x'_2}) \qdim(y) p_z = \qtr_x(a_x) \qdim(y) p_z .
  \end{align*}

  Now we make the sum over $y$ and $z$. Since we are in the case where less than $n-|x|$ letters are simplified in the tensor product $y\otimes z$, the terms where $y$ does not start with $x$ vanish and we get
  \begin{displaymath}
    \sum_{l=0}^{n-|x|}(\qTr_n\otimes p_r)(b_s) \leq \qtr_x(a_x) p_r \times {\ts\sum} \{\qdim(y) ; |y| = n, y\geq x\}.
  \end{displaymath}
  Finally since $\nu$ is a state we have $\nu\psi_{r,\infty}(p_r)\leq 1$ and we obtain, using the notation $\qdim(x,n)$ from the previous lemma:
  \begin{displaymath}
    \sum_{l=0}^{n-|x|}(\qtr_n\otimes \nu\psi_{r,\infty})(b_s) \leq \qtr_x(a_x) \frac{\qdim(x,n)}{\qdim(n)}.
  \end{displaymath}
  Note that as $n\to\infty$ the right-hand side tends to $\qtr_x(a_x)\omega_I(\partial I(x)) = \omega(\psi_{x,\infty}(a_x))$.

  \bigskip

  Altogether~\eqref{eq unique stat proof} yields
  $\nu(\psi_{x,\infty}(a_x)) \leq 2^{-k}\epsilon + |x| 2^{-2k} +
  \omega(\psi_{x,\infty}(a_x))$ for any $x = x_1x_2$, $x_2 \in I'_k$. Summing
  over $x'_2$, which takes less than $2^k$ values, this yields
  $\nu(a)\leq \epsilon + (|x_1|+k)2^{-k} + \omega(a) + 2\epsilon \leq
  \omega(a)+4\epsilon$. This is true for any $\epsilon>0$, hence
  $\nu(a)\leq\omega(a)$. This is true for every
  $a\in \psi_{x_1,\infty}(B(H_{x_1})_+)$, hence $\nu=\omega$ on
  $\psi_{x_1,\infty}(B(H_{x_1}))$ for every $x_1$, and finally $\nu = \omega$.
\end{proof}

\begin{cor} \label{Cor FUF boundary}
  If $N\geq 3$, the Gromov boundary of $\F U_F$ is an $\F U_F$-boundary.  
\end{cor}

\begin{proof}
  It is already known that the Poisson map $P_1 = (\id\otimes\omega)\beta : C(\partial_G\F U_F) \to H^\infty(\F U_F,\qtr_1) \subset \ell^\infty(\F U_F)$ is
  completely isometric \cite{VVV10}. Taking into account the unique stationarity result of Theorem~\ref{Unique stationary state}, we can thus apply
  \cite[Theorem~4.19]{KKSV22} which shows that that $\partial_G\F U_F$ is an $\F U_F$-boundary.
\end{proof}

\section*{Acknowledgments}

The authors would like to thank Fatemeh Khosravi for some interesting discussions and, as well, for pointing out the references \cite{PZ15, Z19}. R.V. and B.A.-S. were partially supported by the ANR grant ANR-19-CE40-0002. B.A.-S. was partially supported by PIMS and the Simons Foundation grant PPTW GR023618.

{\small\bibliography{article}}
\bibliographystyle{alpha}
\renewcommand*{\bibname}{References}

\end{document}